\documentclass[10pt,reqno]{amsart}
\usepackage{amsmath, latexsym, amsfonts, amssymb,amsthm, amscd,epsfig,enumerate}
\usepackage{subfigure,color}

\advance\hoffset -.75cm

\usepackage[usesnames]{xcolor}

\definecolor{refkey}{gray}{.75}
\definecolor{labelkey}{gray}{.75}

\newcommand{\N}{\mathbb N}
\newcommand{\Prob}{\mathbb P}

\newcommand{\cF}{\mathcal{F}}

\newcommand{\cH}{\mathcal{H}}

\newcommand{\pr}{\mathbb P}
\newcommand{\T}{\mathbb T}

\newcommand{\ident}{{\mathchoice {\rm 1\mskip-4mu l} {\rm 1\mskip-4mu l}
{\rm 1\mskip-4.5mu l} {\rm 1\mskip-5mu l}}}

\newtheorem{teo}{Theorem}[section]
\newtheorem{lem}[teo]{Lemma}
\newtheorem{cor}[teo]{Corollary}
\newtheorem{rem}[teo]{Remark}
\newtheorem{pro}[teo]{Proposition}
\newtheorem{defn}[teo]{Definition}
\newtheorem{exmp}[teo]{Example}

\newtheorem{assump}[teo]{Assumption}

\title
{Strong local survival of branching random walks \\ is not monotone}

\author[D.~Bertacchi]{Daniela Bertacchi}
\address{D.~Bertacchi, Dipartimento di Matematica e Applicazioni,
Universit\`a di Milano--Bicocca,
via Cozzi 53, 20125 Milano, Italy.}
\email{daniela.bertacchi\@@unimib.it}

\author[F.~Zucca]{Fabio Zucca}
\address{F.~Zucca, Dipartimento di Matematica,
Politecnico di Milano,
Piazza Leonardo da Vinci 32, 20133 Milano, Italy.}
\email{fabio.zucca\@@polimi.it}

\date{}

\begin{document}

\begin{abstract}
The aim of this paper is the study of the  strong local survival property for
 discrete-time and continuous-time branching random walks. We study this property
by means of an infinite dimensional generating function $G$ and a 
maximum principle which, we prove, is satisfied
by every fixed point of $G$. 
We give results about the existence of a strong local survival regime and 
we prove that, unlike local and global survival, in continuous time, strong local survival is not a monotone property
in the general case (though it is monotone if the branching random walk is quasi transitive). 
We provide an example of an irreducible branching random walk
where the strong local property depends on the starting site of the process.
By means of other counterexamples we show that the existence of a pure global phase is not equivalent
to nonamenability of the process, and that even an irreducible branching random walk with the same branching law
at each site may exhibit non-strong local survival. Finally we show that the generating function
of a irreducible BRW can have more than two fixed points; this disproves a previously known result.
\end{abstract}


\maketitle


\section{Introduction}
\label{sec:intro}

%
A branching process is a very simple population model (introduced in \cite{cf:GW1875})
where particles breed and die (independently of each other) according to some random law.
At any time, this process is completely characterized by the total number of particles alive.
Branching random walks (in short, BRWs) add space to this picture: particles live in a spatially structured environment
and the reproduction law, which may depend on the location, not only tells how many children the particle
has, but also where it places them. The state of the process, at any time, is thus described by the
collection of the numbers of particles alive at $x$, where $x$ varies among the possible sites.
%
In the literature one can find BRWs both in continuous and discrete time.
The continuous-time setting has been studied by many authors
(see \cite{cf:HuLalley, cf:Ligg1, cf:Ligg2, cf:MadrasSchi, cf:MountSchi} just to name a few) along with some
variants of the process (see \cite{cf:BL, cf:BBZ, cf:BLZ, cf:BPZ, cf:BZ3}).
The discrete-time case has been initially considered as a natural
generalization of branching processes
(see \cite{cf:AthNey, cf:Big1977, cf:Big1978, cf:BigKypr97, cf:BigRah05, cf:Harris63}).
The definition of discrete-time BRW that we give in Section~\ref{subsec:discrete}
is sufficiently general to include the discrete-time counterpart that 
every continuous-time BRW admits.
Since every continuous-time BRW and its discrete-time counterpart have the same asymptotic
behavior, it suffices to provide results for the discrete-time case. On the other hand,
continuous-time examples naturally yield discrete-time ones.
Our definition also includes as particular cases: BRWs with independent diffusion
(where particles are first generated and then dispersed independently according to
a diffusion matrix $P$, see Section~\ref{subsec:discrete} and equation~\eqref{eq:particular1});
BRWs with no death (where each particle has null probability of having no children);
BRWs whose total number of particles behaves as a branching process (where the law of the number
of offspring does not depend on the site, we call these BRWs locally isomorphic to a branching process,
see Section~\ref{subsec:FBRWs}).

The basic question which arises studying the BRW is whether it survives with positive 
probability and, in this case, if it visits a given  site infinitely many times. 
The first question
asks whether there is global survival, that is, with
positive probability at any time there is someone alive \textit{somewhere}); while the second question
deals with local survival, that is, whether with
positive probability the process returns infinitely many times to some fixed sites.
It is clear that the probability of global survival is larger or equal to the probability of
local survival. If the probability of global survival is strictly larger than the one of local
survival, then the latter may be positive or null. In the first case, we say that there is non-strong local survival, in the second case there is pure global survival.
When on the contrary, the probabilities of global and local survival
are equal and strictly positive, we say that the BRW has strong local survival.
Hence, strong local survival means that the events of local and global survival coincide
(but for a null probability set) and have positive probability.

The interest on the strong local behavior is fairly recent
(see for instance \cite{cf:GMPV09, cf:Muller08-2}).
The aim of this paper is to study some properties of the strong local survival, comparing
them with the corresponding ones of local and global survival.

As in the case of branching processes, the main tool is that probabilities of extinction are fixed points of 
an infinite-dimensional generating function $G$ (see Sections~\ref{subsec:genfun} for the definition
and \ref{subsec:survivalprob} for its link to the extinction probabilities).
It is worth noting that, unlike the branching process case, it is not true that $G$ has at most
two fixed points, even in the irreducible case (where with positive probability a particle
at site $x$ can have a progenies at site $y$, for all $x$ and $y$). 
Indeed, we prove this indirectly by providing examples of irreducible BRWs which survive locally
but with a smaller probability than the one of global survival (hence non-strong locally, see 
Examples~\ref{ex:nonstronglocalFBRW} and \ref{ex:nonstronglocalFBRW2}) and directly by an explicit construction
of three fixed points for the $G$ of a certain BRW (Remark~\ref{rem:Spataru}).
By Corollary~\ref{cor:finiteBRW} we have that in the irreducible case a
sufficient condition for the existence of at most two fixed points for $G$ is the finiteness of the
set of vertices.

In the particular case where there is no branching,
one gets a random walk and the role of $G$ and its fixed points is played 
by the transition matrix and the harmonic functions, respectively. It is thus natural to look
for a maximum principle in the context of branching random walks as well (see
Proposition~\ref{pro:maximumprinciple}).
As an application, we have that in the irreducible case, pure global survival is independent
of the starting vertex. This is also true for local and global survival, but it does not hold
for strong local survival, unless that the probability of having zero children is positive for all
sites or if the BRW is quasi transitive  (see Sections~\ref{subsec:FBRWs}
and \ref{subsec:survivalprob} and 
Corollary~\ref{pro:qtransitive}). Example~\ref{exm:irreduciblenotstrongeverywhere} shows that
we may have strong local survival starting from some vertices and non-strong local survival
starting from others.

The speed of reproduction of a continuous-time BRW is proportional to a positive parameter
$\lambda$ (see Section~\ref{subsec:continuous}). It is easily seen that the probability of local and 
global survival are nondecreasing functions with respect to $\lambda$; thus local and global
survival are monotone properties (meaning that if one of them holds for some $\lambda_0$
then it holds for all $\lambda \ge \lambda_0$) and it is possible to define the local
and global critical parameters $\lambda_s$ and $\lambda_w$ (see Section~\ref{subsec:continuous}). 
We show that monotonicity in $\lambda$ does not hold for strong local
survival and it is thus impossible in general to define a strong local critical parameter: 
for the irreducible BRW in Example~\ref{rem:nonstrongandstrong}
if $\lambda$ is small enough or large enough there is
strong local survival but in a intermediate interval for $\lambda$ there is global and local
survival with different probabilities. 

Here is the outline of the paper.
 In Section~\ref{sec:basic} we give the necessary definitions and some basic facts about
discrete-time BRWs (Section~\ref{subsec:discrete}), 
continuous-time BRWs (Section~\ref{subsec:continuous}), the infinite-dimensional
generating function $G$, defined on $[0,1]^X$ associated to a BRW (Section~\ref{subsec:genfun}) and the
special class of $\mathcal F$-BRWs (Section~\ref{subsec:FBRWs}). This class contains
properly the class of BRWs on quasi-transitive graphs (which were studied in \cite{cf:Stacey03}).
We also exhibit the explicit expression of $G$ in a particular case of BRW with independent
diffusion (equation~\eqref{eq:Gcontinuous}) and of BRW with no death constructed from a BRW with death
disregarding all particles with finite progenies (equation~\eqref{eq:G-one}).
Moreover in Section~\ref{subsec:genfun} we state a maximum principle
for the solutions of the equation $G(v) \ge v$, including all fixed points of $G$
 (Proposition~\ref{pro:maximumprinciple}).

Section~\ref{sec:survival} is devoted to the study of all the types of survival.
We first recall, in Section~\ref{subsec:local}, results on local and global survival
(Theorems~\ref{th:discretesurv} and \ref{th:continuoussurv}).
In Section~\ref{subsec:survivalprob} extinction probabilities
are seen as fixed points of the generating function $G$. 
Theorem~\ref{rem:strongconditioned} gives equivalent conditions for strong
local survival, in terms of extinction probabilities, which are useful to 
prove that strong local survival is not
monotone. 

From the maximum principle 
we derive
Theorem~\ref{th:Fgraph}, which describes some properties of fixed points of $G$ for $\mathcal F$-BRWs,
and Corollaries~\ref{cor:finiteBRW} and \ref{pro:qtransitive}.
Corollary~\ref{pro:qtransitive} shows that for an irreducible, 
quasi-transitive BRW, there are only three possible behaviours
(independently of the starting vertex):
global extinction, pure global survival or strong local survival.
Thus, in this case 
strong local survival is monotone
and the critical parameter is $\lambda_s$.
A characterization of strong local survival in terms of the existence of a solution of some inequalities
involving the generating function $G$ is given by Theorem~\ref{th:MenshikovVolkov}.

Section~\ref{sec:examples} is devoted to examples and counterexamples.
For a continuous-time irreducible $\mathcal F$-BRW,  the existence of a pure global phase
is equivalent to nonamenability (see Section~\ref{subsec:discrete} and 
Section~\ref{subsec:stronglocal}). 
Nevertheless in general nonamenability neither implies nor is implied by the existence of
a pure global phase (Example~\ref{exm:amenable}). 
Finally we show (Examples~\ref{ex:nonstronglocalFBRW} and
\ref{ex:nonstronglocalFBRW2}) that even fairly simple BRWs (such as irreducible 
BRWs with independent diffusion
and with offspring distribution independent of the site) may have non-strong local survival. 
This implies that, even in the irreducible case, the generating function $G$ may have more than two fixed points
in $[0,1]^X$ and disproves a result in \cite{cf:Spataru89} (see Remark~\ref{rem:Spataru}).

\section{Basic definitions and preliminaries}
\label{sec:basic}

\subsection{Discrete-time Branching Random Walks}
\label{subsec:discrete}

We start with the construction of a generic discrete-time BRW $\{\eta_n\}_{n \in \N}$
(see also \cite{cf:BZ2} where it is
called \textit{infinite-type branching process}) on a set $X$ which is
at most countable; $\eta_n(x)$ represents the number of particles alive
at $x \in X$ at time $n$. To this aim we
consider a family $\mu=\{\mu_x\}_{x \in X}$
of probability measures 
on the (countable) measurable space $(S_X,2^{S_X})$
where $S_X:=\{f:X \to \N\colon \sum_yf(y)<\infty\}$.
To obtain generation $n+1$ from generation $n$ we proceed as follows:
a particle at site $x\in X$ lives one unit of time,
then a function $f \in S_X$ is chosen at random according to the law $\mu_x$
and the original particle is replaced by $f(y)$ particles at
$y$, for all $y \in X$; this is done independently for all particles of
generation $n$ (a similar construction in random environment can be found
in \cite{cf:GMPV09}).
Note that the choice of $f$ assigns simultaneously the total number of children
and the location where they will live.
We denote the BRW by the couple $(X,\mu)$.

Equivalently we could introduce the BRW by choosing first the number of children and afterwards
their location.
Indeed define $\cH:S_X \rightarrow \N$ as $\cH(f):=\sum_{y \in X} f(y)$
which represents the total number of children associated to $f$.
Denote by $\rho_x$ the measure on $\N$ defined by
$\rho_x(\cdot):=\mu_x(\cH^{-1}(\cdot))$; this is the law of the random number of children
of a particle living at $x$. For each particle, independently, we pick a number $n$ at random,
according to the law $\rho_x$, then we choose a function $f \in \cH^{-1}(n)$ with probability
$\mu_x(f)/\rho_x(n)\equiv\mu_x(f)/\sum_{g \in \cH^{-1}(n)}\mu_x(g)$ and
we replace the particle at $x$ with $f(y)$ particles at $y$ (for all $y \in X$).



In BRW theory  a fundamental role is played by the \textit{first-moment matrix}
$M=(m_{xy})_{x,y \in X}$,
where
$m_{xy}:=\sum_{f\in S_X} f(y)\mu_x(f)$ is
 the expected number of particles from $x$ to $y$
(that is, the expected number of children that a particle living
at $x$ sends to $y$).
 We
suppose that $\sup_{x \in X} \sum_{y \in X} m_{xy}<+\infty$; most of the results
of this paper still hold without this hypothesis, nevertheless
it allows us to avoid dealing with an infinite expected number of offsprings.
The expected number of children generated by a particle living at $x$
is $\sum_{y \in X} m_{xy} = \sum_{n \ge 0} n \rho_x(n)=:\bar \rho_x$.
Given a function $f$ defined on $X$ we denote by $Mf$ the function $Mf(x):=\sum_{y \in X} m_{xy}f(y)$
whenever the right-hand side converges absolutely for all $x$.
We denote by
 $m^{(n)}_{xy}$
the entries of the $n$th power matrix $M^n$ and we define
\begin{equation}\label{eq:generalgeomparam}
M_s(x,y) := \limsup_{n \to \infty} \sqrt[n]{m_{xy}^{(n)}}, \quad
M_w(x) := \liminf_{n \to \infty} \sqrt[n]{\sum_{y \in X} m_{xy}^{(n)}}, \qquad \forall x,y \in X.
\end{equation}

Explicit computations of  $M_s(x,y)$ and $M_w(x)$ are possible in some cases (see \cite{cf:BZ, cf:BZ2}):
in particular $M_s(x,x)$ can be obtained by means of a generating function (see  \cite[Section 3.2]{cf:Z1}).
In this paper we do not need to compute explicitly $M_s$ and $M_w$ except for some specific examples
where justifications will be provided. 

For a generic BRW, we call \textit{diffusion matrix} the matrix $P$ with entries $p(x,y)=m_{xy}/\bar \rho_x$.
In particular if $\bar \rho_x$ does not depend on $x \in X$, we have that $M_w(x)=\bar \rho$ for all $x \in X$ and
$M_s(x,y)=\bar \rho \cdot \limsup_{n \to \infty} \sqrt[n]{p^{(n)}(x,y)}$ (where the $\limsup$ defines
the \textit{spectral radius} of $P$ according to \cite[Chapter I, Section 1.B]{cf:Woess}). 

Note that, in the general case, the locations of the offsprings are not chosen independently
(they are assigned by the chosen $f\in S_X$).
When
the offsprings are dispersed independently according to $P$ we call the process a
\textit{BRWs with independent diffusion}:
in this case
\begin{equation}\label{eq:particular1}
\mu_x(f)=\rho_x \left (\sum_y f(y) \right )\frac{(\sum_y f(y))!}{\prod_y f(y)!} \prod_y p(x,y)^{f(y)},
\quad \forall f \in S_X.
\end{equation}

To  a generic discrete-time BRW we associate a graph $(X,E_\mu)$ where $(x,y) \in E_\mu$  
if and only if $m_{xy}>0$.
We denote by $\mathrm{deg}(x)$
the degree of a vertex $x$, that
is, the cardinality of the set $ \mathcal{N}_x:=\{y\in X\colon  (x,y) \in E_\mu
\}$.
We say that there is a path from $x$ to $y$, and we write $x \to y$, if it is
possible to find a finite sequence $\{x_i\}_{i=0}^n$ (where $n \in \N$)
such that $x_0=x$, $x_n=y$ and $(x_i,x_{i+1}) \in E_\mu$
for all $i=0, \ldots, n-1$. If $x \to y$ and $y \to x$ we write $x \rightleftharpoons y$.
Observe that there is always a path of length $0$ from $x$ to itself. The equivalence 
class $[x]$ of $x$ with respect to $\rightleftharpoons$ is called \textit{irreducible class of $x$}.
It is easy to show that if $x \rightleftharpoons x^\prime$ and $y \rightleftharpoons y^\prime$ then
$M_s(x,y)=M_s(x^\prime,y^\prime)$ and $M_w(x)=M_w(x^\prime)$. Moreover, 
$m^{(n)}_{xx}$ and $M_s(x,x)$ depend only on the entries $(m_{ww^\prime})_{w,w^\prime \in [x]}$. 
We call the matrix $M=(m_{xy})_{x,y \in X}$ \textit{irreducible} if and only if
the graph $(X,E_\mu)$ is \textit{connected} (that is, there is only one irreducible class), 
otherwise we call it \textit{reducible}.
From the BRW point of view, the irreducibility of $M$ means that the progeny of any particle can spread to any site of the 
graph. For an irreducible BRW, $M_s(x,y)=M_s$ and $M_w(x)=M_w$ for all $x,y \in X$.

The BRW $(X,\mu)$ is called \textit{non-oriented} or \textit{symmetric} if $m_{xy}=m_{yx}$ for every $x,y \in X$.
Note that if $(X,\mu)$ is non-oriented then the graph $(X,E_\mu)$ is non-oriented (that is, $(x,y) \in E_\mu$
if and only if $(y,x) \in E_\mu$).
$(X,\mu)$ is called \textit{nonamenable} if and only if
\[ 
\inf
\left \{
\frac{\sum_{x \in S, y \in S^\complement} m_{xy}}{|S|}\colon  S \subseteq X, |S| < \infty
\right \}
>0,
\] 
and it is called \textit{amenable} otherwise.

The idea behind the definition of nonamenability is that the expected number of children placed outside every finite subset of 
$X$ is always comparable with the size of the subset itself. This suggests, in principle, that it should be possible 
for the BRW to survive and, at the same time, to escape from every finite set.
This is true for a subclass of BRWs but not in general, see Example~\ref{exm:amenable} and the preceding discussion.
We note that, if $m_{xy}\in\{0, \lambda\}$ (for some fixed $\lambda>0$) then the BRW is nonamenable if and only if the graph $(X, E_\mu)$ is  
nonamenable according to the usual definition for graphs (see \cite[Chapter II, Section 12.B]{cf:Woess}).
%
%
%


Depending on the initial configuration, the process can survive in different ways.
We consider initial configurations with only one particle placed at a fixed site $x$:
let $\pr^{\delta_{x}}$ be the law of this process. Throughout this paper \textit{wpp}
is shorthand for ``with positive probability''.

\begin{defn}\label{def:survival} $\ $
\begin{enumerate}
 \item 
The
process \textsl{survives locally wpp} in $A \subseteq X$ starting from $x \in X$
if
$
{\mathbf{q}}(x,A):=1-\pr^{\delta_x}(\limsup_{n \to \infty} \sum_{y \in A} \eta_n(y)>0)<1.
$
\item
The process \textsl{survives globally wpp} starting from $x$ if
$
\bar {\mathbf{q}}(x):={\mathbf{q}}(x,X) 
<1.
$
\item
There is \textsl{strong local survival wpp} in $A \subseteq X$ starting from $x \in X$
if
$ 
{\mathbf{q}}(x,A)=\bar {\mathbf{q}}(x)<1
$ 
and \textsl{non-strong local survival wpp} in $A$ if $\bar {\mathbf{q}}(x)<{\mathbf{q}}(x,A)<1$.
\item
The BRW is in a \textit{pure global survival phase} starting from $x$ if
$ 
\bar {\mathbf{q}}(x)<{\mathbf{q}}(x,x)=1
$ 
(where we write ${\mathbf{q}}(x,y)$ instead of ${\mathbf{q}}(x, \{y\})$ for all $x,y \in X$).
\end{enumerate}
\end{defn}
\noindent
From now on, when we talk about survival, ``wpp'' will be tacitly understood.
Often we will say simply that local survival occurs ``starting from $x$'' or ``at $x$'':
in this case we mean that $x=y$. When there is no survival wpp, we say that there is extinction
and the fact that extinction occurs with probability one will be tacitly understood.

Note that ${\mathbf{q}}(x,A)$ are the probabilities of extinction in $A$ starting from $x$.
Roughly speaking, 
there is strong survival at $y$ starting from $x$ if and only if the probability
of local survival at $y$ starting from $x$ conditioned on global survival starting from $x$ is $1$.
Thus, strong local survival means that for almost all realizations the process either survives locally
(hence globally) or it goes globally extinct. 
There are many relations between $\bar {\mathbf{q}}(x)$ and ${\mathbf{q}}(x,y)$ and between ${\mathbf{q}}(w,x)$ and
${\mathbf{q}}(w,y)$ where $x,y, w \in X$ (see for instance Section~\ref{subsec:survivalprob} or \cite{cf:BZ4, cf:Z1}).

In order to avoid trivial situations where particles have one offspring almost surely, we assume
henceforth the following.
\begin{assump}\label{assump:1}
For all $x \in X$ there is a vertex $y \rightleftharpoons x$ such that
$\mu_y(f\colon  \sum_{w\colon w \rightleftharpoons y} f(w)=1)<1$,
 that is, in every equivalence class (with respect to $\rightleftharpoons$)
there is at least one vertex where a particle
can have inside the class a number of children different from one wpp.
\end{assump}

\subsection{Continuous-time Branching Random Walks}
\label{subsec:continuous}


In continuous time each particle has an exponentially distributed
random lifetime with parameter 1. The breeding mechanisms can
be regulated by means of a nonnegative matrix $K=(k_{xy})_{x,y \in X}$ in such a way that
for each particle alive at $x$,
there is a clock with $Exp(\lambda k_{xy})$-distributed intervals (where $\lambda>0$),
each time the clock
rings the particle places one son at $y$. We say that the BRW has a death rate 1 and a
reproduction rate $\lambda k_{xy}$ from $x$ to $y$.
We observe (see Remark~\ref{rem:deathrate}) that the assumption
of a nonconstant death rate does not represent a
significant generalization.
We denote by  $(X,K)$ a family of continuous-time BRWs (depending on the parameter $\lambda>0$),
while we use the notation $(X,\mu)$ for
a discrete-time BRW.


To a continuous-time BRW
one can associate a discrete-time counterpart which takes into account all the offsprings 
of a particle before it dies;
in this sense the theory of continuous-time BRWs, as long as 
as it concerns the probabilities of survival (local, strong local and global),
is a particular case of the theory of discrete-time BRWs.
Elementary calculations show that $\mu_x$ satisfies 
equation~\eqref{eq:particular1}, where
\begin{equation}\label{eq:counterpart}
\rho_x(i)=\frac{1}{1+\lambda k(x)} \left ( \frac{\lambda k(x)}{1+\lambda k(x)} \right )^i, \qquad
p(x,y)=\frac{k_{xy}}{k(x)},
\end{equation}
($k(x):=\sum_{y \in X} k_{xy}$).
Note that the discrete-time
counterpart of a continuous-time BRW is a BRW with independent diffusion and that
$\rho_x$ depends only on $\lambda k(x)$.
It is straightforward to show that $m_{xy}=\lambda k_{xy}$ and $\bar \rho_x=\lambda k(x)$.
Moreover equation~\eqref{eq:counterpart} shows that
the discrete-time
counterpart satisfies Assumption~\ref{assump:1}.
%
All the definitions given in the discrete-time case extend to continuous-time BRWs:
a continuous-time BRW has some property if and only if its discrete-time
counterpart has it.

\begin{rem}\label{rem:deathrate}
The same construction applies to continuous-time BRWs with a death rate $d(x)>0$ dependent on $x \in X$. In this
case the discrete-time counterpart satisfies equation~\eqref{eq:particular1} where
\[ 
\rho_x(i)=\frac{d(x)}{d(x)+\lambda k(x)} \left ( \frac{\lambda k(x)}{d(x)+\lambda k(x)} \right )^i, \qquad
p(x,y)=\frac{k_{xy}}{k(x)}.
\] 
Hence, from the point of view of local and global survival, this process is equivalent to a continuous-time BRW with death rate
$1$ and reproduction rate $\lambda k_{xy}/d(x)$ from $x$ to $y$.
\end{rem}

%
%

Given $x \in X$, two critical parameters are associated to the
continuous-time BRW: the \textit{global} 
\textit{survival critical parameter} $\lambda_w(x)$ and the  \textit{local} 
 \textit{survival critical parameter} $\lambda_s(x)$.
They are defined as
\[ 
\begin{split}
\lambda_w(x)&:=\inf \Big \{\lambda>0\colon \,
\pr^{\delta_{x}}\Big (\sum_{w \in X} \eta_t(w)>0, \forall t\Big) >0 \Big \},\\
\lambda_s(x)&:=
\inf\{\lambda>0\colon \,
\pr^{\delta_{x}} \big(\limsup_{t \to \infty} \eta_t(x)>0 \big) >0
\}.
  \end{split}
\] 
These values are constant in every irreducible class; in particular they do not depend
on $x$ if the BRW is irreducible.
The process is called
\textit{globally supercritical}, \textit{critical} or \textit{subcritical}
if $\lambda>\lambda_w$, $\lambda=\lambda_w$ or $\lambda<\lambda_w$;
an analogous definition is given for the local behavior using $\lambda_s$ instead of $\lambda_w$.
In particular we say that there exists a \textit{pure global survival phase} starting from $x$ if the
interval $(\lambda_w(x),\lambda_s(x))$ is not empty; clearly, if $\lambda  \in (\lambda_w(x),\lambda_s(x))$
then the BRW is in a pure global survival phase according to Definition~\ref{def:survival}.

Given a continuous-time BRW $(X,K)$ we define 
\[ 
K_s(x,y) := \frac{M_s(x,y)}\lambda\equiv\limsup_{n \to \infty} \sqrt[n]{k_{xy}^{(n)}}, \quad
K_w(x) := \frac{M_w(x)}\lambda\equiv\liminf_{n \to \infty} \sqrt[n]{\sum_{y \in X} k_{xy}^{(n)}}, \qquad \forall x,y \in X,
\] 
where $M_s(x,y)$ and $M_w(x)$ are the corresponding parameters of the discrete-time counterpart.
$K_s(x,y)$ and $K_w(x)$ depend only on the equivalence classes of $x$ and $y$, hence
if the BRW is irreducible, then they do not depend on $x,y \in X$.
%
%
%

We say that a BRW is \textit{site-breeding} if $k(x)$ does not depend on $x \in X$.
We say that a BRW is \textit{edge-breeding} if $k_{xy} \in \N$. The typical edge-breeding BRW can be
constructed from a multigraph with set of vertices $X$ by defining $k_{xy}$ as
the number of edges from $x$ to $y$; in this case to each edge there corresponds a constant
reproduction rate $\lambda$.
If the multigraph is a graph, then it coincides with the graph $(X,E_\mu)$ associated 
with the discrete-time counterpart of the edge-breeding BRW.

\subsection{Infinite-dimensional generating function}\label{subsec:genfun}

We associate a generating function $G:[0,1]^X \to [0,1]^X$
to the family $\{\mu_x\}_{x \in X}$
which can be considered as an infinite dimensional power series. 
More precisely,
for all ${\mathbf{z}} \in [0,1]^X$, $G({\mathbf{z}}) \in [0,1]^X$ is defined as the following weighted sum of (finite) products
\[ 
G({\mathbf{z}}|x):= \sum_{f \in S_X} \mu_x(f) \prod_{y \in X} {\mathbf{z}}(y)^{f(y)},
\] 
where $G({\mathbf{z}}|x)$ is the $x$ coordinate of $G({\mathbf{z}})$.
The family $\{\mu_x\}_{x \in X}$ is uniquely determined by $G$. Indeed fix a finite $X_0 \subseteq X$ and $x \in X$.
For every $\mathbf{z}$ with support in $X_0$, we have  
$G({\mathbf{z}}|x)= \sum_{f \in S_{X_0}} \mu_x(f) \prod_{y \in X_0} {\mathbf{z}}(y)^{f(y)}$
which can be identified with a power series with several variables (defined on $[0,1]^{X_0}$).
Suppose now we have another generating function $ \overline G$ (associated to  $\{\overline \mu_x\}_{x \in X}$)
such that $G=\overline G$. In particular $G({\mathbf{z}}|x)=\overline G({\mathbf{z}}|x)$ for every 
$\mathbf{z}$ with support in $X_0$. Thus $\mu_x(f)=\overline \mu_x(f)$ for all $f \in S_{X_0}$.
Since $S_X =\bigcup_{\{X_0 \subseteq X \colon X_0 \textrm{ finite}\}}S_{X_0}$ we have that
$\mu_x(f)=\overline \mu_x(f)$ for
all $f \in S_X$.

Note that $G$ is continuous with respect to the \textit{pointwise convergence topology} of $[0,1]^X$  and nondecreasing
with respect to the usual partial order of $[0,1]^X$ (see \cite[Sections 2 and 3]{cf:BZ2} for further details).
Moreover, $G$ represents the 1-step reproductions; we denote by $G^{(n)}$ the generating function
associated to the $n$-step reproductions, which is inductively defined as $G^{(n+1)}({\mathbf{z}})=G^{(n)}(G({\mathbf{z}}))$,
where $G^{(0)}$ is the identity.
Extinction probabilities are fixed points
of $G$ and the smallest one is $\bar {\mathbf{q}}$ (see Section~\ref{subsec:survivalprob} for details). 

%
%
An example where the function $G$ can be explicitly computed is a BRW 
with independent diffusion: in this case it is not difficult to see that
$G({\mathbf{z}}|x)=\sum_{n \in \N} \rho_x(n) (P{\mathbf{z}}(x))^n$
where $P{\mathbf{z}}(x)=\sum_{y \in X} p(x,y){\mathbf{z}}(y)$.
If, in particular, $\rho_x(n)=\frac{1}{1+\bar \rho_x} (\frac{\bar \rho_x}{1+\bar \rho_x} )^n$ 
(as in the discrete-time counterpart of a continuous-time BRW) then
the previous expression becomes $G({\mathbf{z}}|x)=(1+\bar \rho_x P(\mathbf{1}-\mathbf{z})(x))^{-1}$.
The previous equality can be written in a more compact way as
\begin{equation}\label{eq:Gcontinuous}
 G({\mathbf{z}})= \frac{\mathbf{1}}{\mathbf{1}+M(\mathbf{1}-{\mathbf{z}})}
\end{equation}
where $M$ is the first-moment matrix and $M \mathbf{v}(x)=\bar \rho_x P\mathbf{v}(x)$ 
(by definition of $P$).
In equation~\eqref{eq:Gcontinuous} and hereafter, whenever ${\mathbf{z}}, {\mathbf{v}} \in [0,1]^X$ the ratio ${\mathbf{z}}/{\mathbf{v}}$ will be taken coordinatewise,
that is, $({\mathbf{z}}/{\mathbf{v}})(x):={\mathbf{z}}(x)/{\mathbf{v}}(x)$ for all $x$ such that ${\mathbf{v}}(x)>0$ 
(the value of $({\mathbf{z}}/{\mathbf{v}})(x)$
if ${\mathbf{v}}(x)=0$, if any, will be explicitly  defined when needed). 

When one is interested in the question whether a global surviving BRW survives strong locally,
it may be useful to condition the process on global survival. 
Given a generic discrete-time BRW such that $\bar {\mathbf{q}}(x)<1$ for all $x \in X$, 
by conditioning on global survival, we associate a BRW with no death
(that is, a BRW such that $\rho_x(0)=0$).
Let $\{\eta_n\}_{n \in \N}$ be the original BRW. Consider the event $\Omega_\infty=
\{\sum_{x \in X} \eta_n(x)>0, \, \forall n \in \N\}$ and define
the process $\{\widehat \eta_n\}_{n \in \N}$ as follows:
$\widehat \eta_n(x, \omega)$ equals the number of particles in $\eta_n(x, \omega)$ with at least one infinite line of descent
when $\omega \in \Omega_\infty$ and it equals  $0$ when $\omega \not \in \Omega_\infty$.
Roughly speaking, $\{\widehat \eta_n\}_{n \in \N}$ is obtained
by $\{\eta_n\}_{n \in \N}$ by removing all the particles with finite progeny, which
are clearly irrelevant in view of the survival due to the fact that $\bar {\mathbf{q}}(x)<1$ for all $x \in X$.
%
%
Hence, we have that the probability of local survival of
$\{\widehat \eta_n\}_{n \in \N}$ in $A$ (for all $A \subseteq X$),
starting from $x$ is equal to the same probability for $\{\eta_n\}_{n \in \N}$,
that is, ${\mathbf{q}}(x,A)$.
It can be shown that this process, restricted to $\Omega_\infty$ is a BRW 
that we call the \textit{no-death BRW associated to $\{\eta_n\}_{n \in \N}$}
(we still denote it by
$\{\widehat \eta_n\}_{n \in \N}$). Its
generating function is
\begin{equation}\label{eq:G-one}
\widehat G({\mathbf{z}}|x) = \frac{G(v({\mathbf{z}})|x)-\bar {\mathbf{q}}(x)}{1-\bar {\mathbf{q}}(x)},
\end{equation}
where 
$G$ is the generating function of the original BRW and $v:[0,1]^X \rightarrow [0,1]^X$ is
defined as $v({\mathbf{z}}|x):=\bar {\mathbf{q}}(x)+{\mathbf{z}}(x)(1-\bar {\mathbf{q}}(x))$.
In a more compact way equation~\eqref{eq:G-one} can be written as $\widehat G=T^{-1}_{\bar {\mathbf{q}}} \circ G \circ T_{\bar {\mathbf{q}}}$ where
$T_{\mathbf{w}}:[0,1]^X \to \{{\mathbf{z}} \in [0,1]^X\colon  {\mathbf{w}} \le {\mathbf{z}}\}$ is defined as $T_{\mathbf{w}}{\mathbf{z}}(x):={\mathbf{z}}(x)(1-{\mathbf{w}}(x))+{\mathbf{w}}(x)$; note that
$T_{\mathbf{w}}$ is nondecreasing and, if ${\mathbf{w}}(x)  <{1}$ for all $x \in X$, bijective.
In particular if $\bar {\mathbf{q}} < \mathbf{1}$ then $T_{\bar {\mathbf{q}}}$ is a bijective map from the set
of fixed points of $\widehat G$ to the set of fixed points of $G$.

Clearly, for all $A \subseteq X$, the probability of local survival in $A$ of the associated no-death BRW
 starting from $x$ is 
the probability of local survival in $A$ of the original BRW  conditioned on 
global survival (starting from $x$), that is, 
$1-(T_{\bar {\mathbf{q}}}^{-1} {\mathbf{q}}(\cdot,A))(x)= (1-{\mathbf{q}}(x,A))/(1-\bar {\mathbf{q}}(x))$.

The following proposition is a sort of \textit{maximum principle} for
the function $({\mathbf{z}}-\bar {\mathbf{q}})/(\mathbf{1}-\bar {\mathbf{q}})$ where ${\mathbf{z}}$ is such that $G({\mathbf{z}}) \ge {\mathbf{z}}$
(note that we are not assuming that $\bar {\mathbf{q}}(x)<1$ for all $x \in X$).

\begin{pro}\label{pro:maximumprinciple} 
Given ${\mathbf{z}} \in [0,1]^X$ such that ${\mathbf{z}} \ge \bar {\mathbf{q}}$ is 
a solution of the inequality $G({\mathbf{z}}) \ge {\mathbf{z}}$,
we define
$\widehat {\mathbf{z}}:=({\mathbf{z}}-\bar {\mathbf{q}})/(\mathbf{1}-\bar {\mathbf{q}})$ where
 $\widehat {\mathbf{z}}(x):=1$ for all
$x$ such that $\bar {\mathbf{q}}(x)=1$. 
Then for all $x \in X$ such that the set $\mathcal{N}_x=\{y\colon (x,y) \in E_\mu\}$
is not empty,
either $\widehat {\mathbf{z}}(y) = \widehat {\mathbf{z}}(x)$ for all $y \in \mathcal{N}_x$ or there exists
$y \in  \mathcal{N}_x$ such that $\widehat {\mathbf{z}}(y)> \widehat {\mathbf{z}}(x)$. In particular
if $\widehat {\mathbf{z}}(x)=1$ then for all $y \in \mathcal{N}_x$ we have $\widehat {\mathbf{z}}(y)=1$.
The same results hold if we take the set $\{y \in X \colon  x \to y\}$ instead of $\mathcal{N}_x$.
\end{pro}

As an application, in a finite, final irreducible class
(for instance if the BRW is irreducible and the set $X$ is finite) if $\mathbf{z}$ is as in Proposition~\ref{pro:maximumprinciple}, then
$\widehat {\mathbf{z}}$ is a
constant vector.

\subsection{$\mathcal F$-BRWs}\label{subsec:FBRWs}

Some results can be achieved if the BRW has some regularity; to this aim
we introduce the concept of $\mathcal F$-BRW (see also \cite[Definition 4.2]{cf:Z1}),
which extends the concept of quasi-transitivity (see below).

\begin{defn}
\label{def:locallyisomorphic} 
$\ $
\begin{enumerate}
 \item 
A BRW $(X, \mu)$ is locally isomorphic to a BRW $(Y,\nu)$ if there exists a
surjective map $g:X\to Y$ such that
$
\nu_{g(x)}(\cdot)=
\mu_x\left(\pi_g^{-1}(\cdot)\right)
$, 
where $\pi_g:S_X \rightarrow S_Y$ is defined as $\pi_g(f)(y)=\sum_{z\in g^{-1}(y)}f(z)$ for all $f\in S_X$, $y \in Y$.
\item
$(X, \mu)$ is a $\mathcal F$-BRW if it is locally isomorphic to some
BRW $(Y,\nu)$ on a finite set $Y$.
\end{enumerate}
\end{defn}
Clearly, if $(X,\mu)$ is locally isomorphic to $(Y, \nu)$ then
\begin{equation} \label{eq:Gfunctions}
G_X({\mathbf{z}} \circ g|x)=G_Y({\mathbf{z}}|g(x))
\end{equation}
for all ${\mathbf{z}} \in [0,1]^Y$ and $x \in X$. 
We note that, since $\mu$ is uniquely determined by $G$,
 equation~\eqref{eq:Gfunctions} holds  if and only if $(X,\mu)$ is locally
isomorphic to $(Y,\nu)$ and $g$ is the map in Definition~\ref{def:locallyisomorphic}.
If $\{\eta_n\}_{n \in \N}$ is a realization of the BRW $(X, \mu)$ then $\{\pi_g(\eta_n)\}_{n \in \N}$ is a realization
of the BRW $(Y, \nu)$.

Using equation~\eqref{eq:Gfunctions} and the fact that $\bar {\mathbf{q}}= \lim_{n \to \infty} G^{(n)}(\mathbf 0)$
(see equation~\eqref{eq:extprobab} with $A=X$), it is possible to prove that
there is global survival for $(X,\mu)$ starting from $x$ if and only if
there is global survival for $(Y,\nu)$ starting from $g(x)$  (see \cite[Theorem 4.3]{cf:Z1}).
%
%
It is not difficult to prove (the details can be found in \cite{cf:Z1} before Theorem 4.3) that, for all $x \in X$, $y \in Y$ and
$n \in \N$, 
$\widetilde m^{(n)}_{g(x)y} = \sum_{z \in g^{-1}(y)} m^{(n)}_{xz}$ where $\widetilde M$ is the first-moment matrix of the BRW $(Y,\nu)$.
This implies $M^{X}_w(x)=M^{Y}_w(g(x))$ for all $x \in X$.

In continuous time (see \cite{cf:BZ2})
one can prove that $(X,K)$ is \textit{locally isomorphic} to $(Y,\widetilde K)$ if and only if
there exists a surjective map $g:X \to Y$ such that
$\sum_{z \in g^{-1}(y)} k_{xz}=\widetilde k_{g(x)y}$
for all
$x \in X$ and $y \in Y$, whence 
%
$K^{X}_w(x)=K^{Y}_w(g(x))$ for all $x \in X$.
In other words, the total rate at which particles at $x$ generate children placing them in the set
of vertices with ``label'' $y$, depends only on $y$ and on $g(x)$.

Roughly speaking, an $\mathcal F$-BRW is a BRW where the vertices of $X$ can be labelled by means of
a finite alphabet $Y$ in such a way that the law of the labels of the positions of the children of 
a particle depends only on the label of the position of the father.  
As an example, consider a graph $(X,E(X))$ such that $\sup_{x \in X} \mathrm{deg}(x)<+\infty$ and where 
$\mathrm{deg}(x)=\mathrm{deg}(y)$ implies $\#\{z \in \mathcal{N}_x \colon \mathrm{deg}(z)=j\}=
\#\{z \in \mathcal{N}_y \colon \mathrm{deg}(z)=j\}$ (for all $j$); an example 
of such a graph is a tree with two alternating degrees.
In this case a BRW on $X$ with independent diffusion where $\rho_x$ depends only on $\mathrm{deg}(x)$  
is an $\mathcal{F}$-BRW and the label
of $x$ is $\mathrm{deg}(x)$.

It is worth mentioning a particular subclass of $\mathcal F$-BRWs:
a BRW is \textit{locally isomorphic to a branching process}
if and only if  the laws of the offspring number $\rho_x=\rho$ is independent
of $x \in X$. 
In this case the BRW is locally isomorphic to a BRW on a singleton $Y:=\{y\}$ where 
the law of the number of children of each particle is $\rho$ and $g(x):=y$ for all $x \in X$.
The explicit computations of $M_w$ and $M_s$ in this case can be found after equation~\eqref{eq:generalgeomparam}.
In particular a continuous-time BRW is 
locally isomorphic to a branching process if and only if $k(x)=k$ for all $x \in X$ 
(that is, if and only if it is an site-breeding BRW). In this case $K_w(x)=k$
and $K_s(x,y)= k\cdot \limsup_{n \to \infty} \sqrt[n]{p^{(n)}(x,y)}$. 


Let $\gamma:X \to X$ be an injective map. 
We say that $\mu=\{\mu_x\}_{x \in X}$ is $\gamma$-invariant if
for all $x
\in X$ and $f \in S_X$ we have
$\mu_x(f)=\mu_{\gamma(x)}( f \circ \gamma^{-1})$ (where $f \circ \gamma^{-1}$ is extended to a function
on $X$ by setting $0$ outside $\gamma(X)$). In particular, a BRW with independent diffusion is $\gamma$-invariant
if and only if $\rho_x=\rho_{\gamma(x)}$ and $p(x,y)=p(\gamma(x),\gamma(y))$ for all $x,y \in X$.

Moreover $(X,\mu)$ is \textit{quasi transitive} if and only if there exists
a finite subset $X_0 \subseteq X$ such that for all $x \in X$ there exists a bijective
map $\gamma:X\to X$ and $x_0 \in X_0$ satisfying $\gamma(x_0)=x$ and $\mu$ is $\gamma$-invariant.
An edge-breeding BRW on a graph $(X,E)$ is quasi transitive if and only if $(X,E)$ is a quasi-transitive graph. 


We note that every quasi-transitive BRW is an $\cF$-BRW (see \cite[Section 6.2]{cf:Z1}). 
The class of $\mathcal{F}$-BRWs is strictly larger than the class of quasi-transitive BRWs.
For instance consider the BRW described in Example~\ref{ex:nonstronglocalFBRW}.
Indeed, in this case 
the BRW is $\gamma$-invariant if and only if, for all $i,j \in \N$,
$p(\gamma(i),\gamma(j))=p(i,j)$. This implies that $\gamma(0)=0$ and, by induction, $\gamma(i)=i$ for all $i \in \N$.
Thus, the only invariant map $\gamma$ is the identity on $\N$, whence there is no finite $X_0$ as described in the definition
of quasi transitivity. Nevertheless, the BRW is locally isomorphic to a branching process, thus it is an $\mathcal{F}$-BRW.
Other examples are the edge-breeding BRWs associated to the following graphs: take a square and attach to each
vertex an infinite branch of a homogeneous tree $\mathbb{T}_3$ of degree 3 (see Figure~\ref{fig:square}); now attach to each vertex of the new graph
an edge with a new endpoint (see Figure~\ref{fig:square-end}). They are both $\mathcal{F}$-BRWs which are not quasi-transitive; moreover while
the first graph is regular (it has constant degree $3$), the second one is not since it has vertices with degree $4$ and vertices with degree
$1$.
\begin{figure}
    \begin{center}
    \subfigure[\tiny Regular graph with degree $3$.]{\includegraphics[height=5cm]{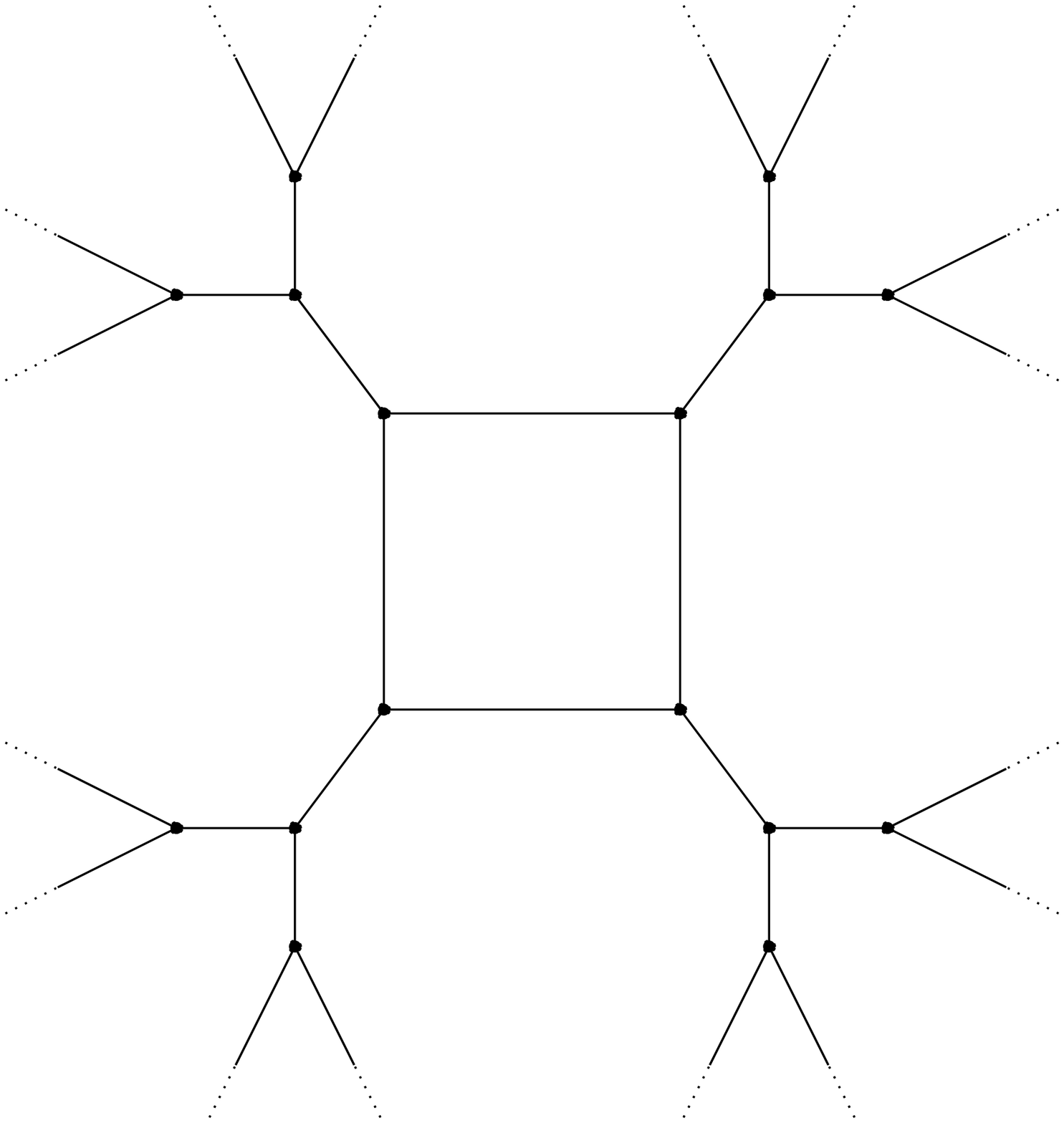}
    \label{fig:square}}
   \hspace{0.5cm}
    \subfigure[\tiny Irregular graph.]{\includegraphics[height=5cm]{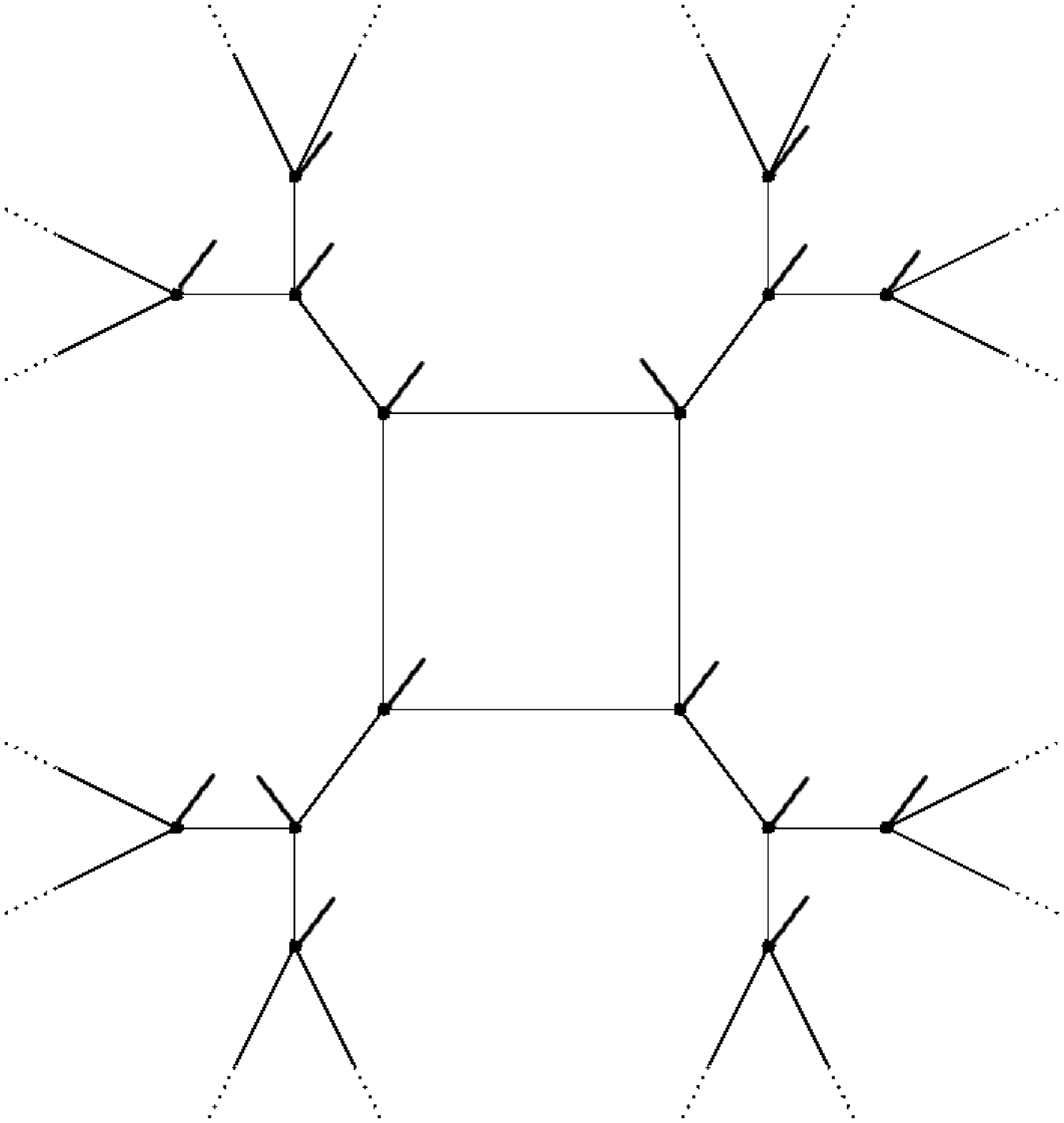}
    \label{fig:square-end}}
    \end{center}
\caption{}
\end{figure}

\section{Conditions for survival and extinction}
\label{sec:survival}

\subsection{Local and global survival}
\label{subsec:local}

The following theorems summarize the main results about local and global survival for discrete-time BRWs and
continuous-time BRWs respectively (see \cite[Theorems 4.1, 4.7, 4.8 and Proposition 4.5]{cf:BZ2}, \cite[Theorems 4.1 and 4.3]{cf:Z1}).  
In particular Theorem~\ref{th:discretesurv}(4) is a straightforward generalization of \cite[Theorem 3.6]{cf:BZ} (we omit the proof). 

\begin{teo}\label{th:discretesurv}
Let $(X,\mu)$ be a discrete-time BRW.
\begin{enumerate}
 \item 
There is
local survival starting from $x$ if and only if $M_s(x,x) 
>1$.
\item
There is global survival starting from $x$ if and only if there exists
${\mathbf{z}}\in [0,1]^X$, ${\mathbf{z}}(x)<1$
such that $G({\mathbf{z}}|y) = {\mathbf{z}}(y)$, for all $y \in X$
(equivalently, such that $G({\mathbf{z}}|y) \le {\mathbf{z}}(y)$, for all $y \in X$).
\item
If $(X,\mu)$ is an $\mathcal F$-BRW then
there is global survival starting from $x$ if and only if $M_w(x)>1$.
\item If $(X,\mu)$ is an irreducible, non-oriented BRW then $M_s < M_w$ if and only if 
$(X,\mu)$ is nonamenable. 
\end{enumerate}
\end{teo}
Note that the fact that there is local survival or not, depends only on the first-moment matrix $M$.
%
%
%
In particular the BRW survives locally at $x$ if and only if it does so when restricted to the irreducibility
class $[x]$ of $x$.
It is worth noting that if $[x]$
is finite, then $M_s(x,x)$
is the Perron-Frobenius eigenvalue
of the submatrix $M^\prime:=(m_{yz})_{y,z \in [x]}$. In this case
there is local survival at $x$ if and only if $\max\{t >0 \colon 
\exists v \not = \mathbf{0}, \, M^\prime v =t v\}>1$.
In general, the global behavior does not depend only on $M$ 
(see \cite[Example 4.4]{cf:Z1}) unless
there is a
one-to-one correspondence
between first moment matrices and processes.
This is true for instance in the class of
BRWs with independent diffusion 
such that  $\rho_x(n)=\frac{1}{1+\bar \rho_x} (\frac{\bar \rho_x}{1+\bar \rho_x} )^n$
(hence for a continuous-time BRW). Indeed 
in that case an equivalent condition
for global survival starting from $x \in X$ is the existence of 
${\mathbf{v}} \in [0,1]^X$, ${\mathbf{v}}(x)>0$ such that
\[
 M{\mathbf{v}} \ge {\mathbf{v}} /(\mathbf{1}-{\mathbf{v}}), \qquad \text{(equivalently, }M{\mathbf{v}} = {\mathbf{v}} /(\mathbf{1}-{\mathbf{v}}) \text{)}
\]
which comes from Theorem~\ref{th:discretesurv}(2) given $\mathbf{z}:=\mathbf{1}-{\mathbf{v}}$ and
the explicit expression \eqref{eq:Gcontinuous} of $G$.
In particular, for a BRW with independent diffusion, the local survival probability $\mathbf{v}_A:=\mathbf{1}-\mathbf{q}(\cdot,A)$
satisfies $M \mathbf{v}_A=\mathbf{v}_A/(\mathbf{1}-\mathbf{v}_A)$, which becomes
$\lambda K  \mathbf{v}_A=\mathbf{v}_A/(\mathbf{1}-\mathbf{v}_A)$ for a continuous-time BRW.

\begin{teo}\label{th:continuoussurv}
Let $(X,K)$ be a continuous-time BRW.
\begin{enumerate}
 \item 
$\lambda_s(x)=1/K_s(x,x)$ and if 
$\lambda=\lambda_s(x)$ then there is local extinction at $x$.
\item
$
 \lambda_w(x) \ge 1/K_w(x).
$
\item If $(X,K)$ is an $\mathcal F$-BRWs then  
$\lambda_w(x)=1/K_w(x)$ and when $\lambda=\lambda_w(x)$ there
is global extinction starting from $x$.
\item If $(X,K)$ is an irreducible, non-oriented $\mathcal{F}$-BRW then $\lambda_s > \lambda_w$ if and only if 
$(X,\mu)$ is nonamenable.
\end{enumerate}
\end{teo}

For a generic BRW when $\lambda=\lambda_w(x)$ there might be global survival
(see \cite[Example 3]{cf:BZ2}).
A characterization of $\lambda_w(x)$ has been given in \cite[Theorem 4.2]{cf:BZ2}
by means of the so-called
 \textit{lower
Collatz-Wielandt number}.

\subsection{Probabilities of extinction and strong local survival}
\label{subsec:survivalprob}

Define ${\mathbf{q}}_n(x,A)$ as the probability of extinction in $A$ no later than the $n$-th
generation starting with one particle at $x$, namely
${\mathbf{q}}_n(x,A)=\Prob(\eta_k(x)=0,\, \forall k\ge n,\,\forall x\in A)$. It is
clear that $\{{\mathbf{q}}_n(x,A)\}_{n \in \N}$ is a nondecreasing sequence satisfying
\begin{equation}\label{eq:extprobab}
 \begin{cases}
 {\mathbf{q}}_{n}(\cdot,A)=G({\mathbf{q}}_{n-1}(\cdot, A)),& \quad \forall n \ge 1\\
 {\mathbf{q}}_0(x,A)=0, &\quad \forall x \in A,
\end{cases}
\end{equation}
hence there is a limit ${\mathbf{q}}(x,A)=\lim_{n \to \infty} {\mathbf{q}}_n(x, A) \in
[0,1]^X$ which is the probability of local extinction in $A$
starting with one particle at $x$ (see
Definition~\ref{def:survival}). Note that equation~\eqref{eq:extprobab}
defines completely the sequence $\{{\mathbf{q}}_n(\cdot,A)\}_{n \in \N}$ only when
$A=X$ (otherwise one needs the values ${\mathbf{q}}_0(x,A)$ for $x \not \in A$).
Since $G$ is continuous we have that ${\mathbf{q}}(\cdot,A)=G({\mathbf{q}}(\cdot,
A))$, hence these extinction probabilities are
fixed points of $G$.

Note that ${\mathbf{q}}(\cdot, \emptyset)= \mathbf 1$.
%
Since $\bar {\mathbf{q}}=\lim_{n \to \infty} G^{(n)}(\mathbf{0})$ we have that
$\bar {\mathbf{q}}$ is the smallest fixed point of $G$ in $[0,1]^X$ (see \cite[Corollary 2.2]{cf:BZ2}).
Using the same arguments, one can prove that $\bar {\mathbf{q}}$ is the smallest
fixed point of $G^{(m)}$ for all $m \in \N$.

Note that $A \subseteq B \subseteq X$ implies ${\mathbf{q}}(\cdot,A)
\ge {\mathbf{q}}(\cdot,B)
\ge \bar {\mathbf{q}}$. Since for all finite $A\subseteq X$ we have ${\mathbf{q}}(x,A) \ge 1-\sum_{y \in A} (1-{\mathbf{q}}(x,y))$
then, for any given finite $A \subseteq X$, ${\mathbf{q}}(x,A)=1$ if and only if ${\mathbf{q}}(x,y)=1$ for all $y \in A$.
%

If $x \to x^\prime$ and $A \subseteq X$ then ${\mathbf{q}}(x^\prime,A)<1$ implies
${\mathbf{q}}(x,A)<1$; as a consequence,
if $x \rightleftharpoons x^\prime$ then
${\mathbf{q}}(x,A)<1$ if and only if ${\mathbf{q}}(x^\prime,A)<1$. Moreover if  $y \rightleftharpoons y^\prime$ 
we have ${\mathbf{q}}(x,y)={\mathbf{q}}(x,y^\prime)$ for all $x \in X$.

In the irreducible case, if $\rho_x(0)>0$ for all $x \in X$, we have that $\bar {\mathbf{q}}(x)={\mathbf{q}}(x,A)$ for
some $x \in X$ and a finite subset $A \subseteq X$ if and only if $\bar {\mathbf{q}}(y)={\mathbf{q}}(y,B)$ for
all $y \in X$ and all finite subsets $B \subseteq X$ (hence, strong local survival is a common property
of all subsets and all starting vertices).
Clearly, this may not be true in the reducible case. 
Besides, if we drop the assumption $\rho_x(0)>0$ for all $x \in X$, we might actually
have $\bar {\mathbf{q}}(x)={\mathbf{q}}(x,A)<1$ and $\bar {\mathbf{q}}(y)<{\mathbf{q}}(y,A)$ for some $x,y \in X$ and a finite $A \subseteq X$
 even when the BRW is irreducible (see Example~\ref{exm:irreduciblenotstrongeverywhere}).
Hence, in general, even for irreducible BRWs, 
strong local survival is not a common property of all vertices
 as local and global survival are.

The following theorem, in the case of global survival, gives equivalent conditions for strong local survival
in terms of extinction probabilities.

\begin{teo}\label{rem:strongconditioned}
We observe that the following assertions are equivalent for every nonempty subset $A \subseteq X$.
\begin{enumerate}[(1)]
\item  ${\mathbf{q}}(x, A) = \bar {\mathbf{q}}(x)$, for all $x \in X$;
\item ${\mathbf{q}}_0(x,A) \le \bar {\mathbf{q}}(x)$, for all $x \in X$;
\item  the probability of visiting $A$ at least once starting from $x$ is larger than or equal to the probability of
global survival starting from $x$, for all $x \in X$:
\item for all $x \in X$, either $\bar {\mathbf{q}}(x)=1$ or
the probability of visiting $A$ at least once starting from $x$ conditioned on global survival starting from $x$ is $1$;
\item  for all $x \in X$, either $\bar {\mathbf{q}}(x)=1$ or
the probability of local survival in $A$ starting from $x$ conditioned on global survival starting from $x$ is $1$
(strong local survival in $A$ starting from $x$).
\end{enumerate}
\end{teo}

From this theorem we have that 
if there exists $x \in X$ such that ${\mathbf{q}}(x,A)>\bar {\mathbf{q}}(x)$ (that is, there is a positive probability of global survival and
nonlocal survival in $A$ starting from $x$) then there exists $y \in X$ such that
${\mathbf{q}}_0(y,A)>\bar {\mathbf{q}}(y)$ (that is, there is a positive probability that the colony survives globally starting from $y$ without
ever visiting $A$). Note that, ${\mathbf{q}}_0(x,A)>\bar {\mathbf{q}}(x)$ implies ${\mathbf{q}}(x,A)>\bar {\mathbf{q}}(x)$ but 
the converse is not true.
In particular for a BRW with no death there is strong local survival in $A$ starting from $x$ for all $x \in X$ if and only if
the probability of visiting $A$ is $1$ starting from every vertex.

We note that, \textit{a priori}, there is no order relation between the events
``visiting $A$ at least once starting from $x$'' and ``global survival starting from $x$''. Nevertheless
if, for all $x \in X$,  the probability of ``visiting $A$ at least once starting from $x$'' is larger than or equal to
the probability of ``global survival starting from $x$'' then, by Theorem~\ref{rem:strongconditioned}
we have that the probability of ``global survival starting from $x$ never visiting $A$'' is $0$
and this implies, whenever $\bar {\mathbf{q}}(x)<1$, that there is strong local survival in $A$
starting from $x$.

In the case of an $\mathcal F$-BRW the fixed-points of $G$
have an interesting property stated in the following theorem.
\begin{teo}\label{th:Fgraph}
 Let $(X, \mu)$ be an $\mathcal F$-BRW.
\begin{enumerate} 
 \item 
There exists at most one fixed point ${\mathbf{z}}$ for $G$ such that $\sup_{x \in X} {\mathbf{z}}(x)<1$,
namely ${\mathbf{z}}=\bar {\mathbf{q}}$.
\item
For all $x \in X$, either
${\mathbf{q}}(\cdot,x) = \bar {\mathbf{q}}(\cdot)$ or $\sup_{w \in X} {\mathbf{q}}(w,x)= 1$. In particular when $(X,\mu)$ is irreducible then
it is either ${\mathbf{q}}(x,x)=\bar {\mathbf{q}}(x)$ for all $x \in X$ or $\sup_{x \in X} {\mathbf{q}}(x,x)= 1$.
\end{enumerate}
\end{teo}
 It is worth noting that, unlike the 
branching process, for a generic irreducible $\mathcal F$-BRW, when $\bar {\mathbf{q}} < \mathbf{1}$,
there might be other fixed points for $G$ (see Examples~\ref{ex:nonstronglocalFBRW}, \ref{ex:nonstronglocalFBRW2} and
Remark~\ref{rem:Spataru}). Nevertheless this cannot happen when $X$ is finite.

\begin{cor}\label{cor:finiteBRW}
If $X$ is finite and the BRW is irreducible then
there are at most two solutions of  $G(\mathbf{z}) \ge \mathbf{z}$ when $\mathbf{z} \ge \bar {\mathbf{q}}$, that is, $\bar {\mathbf{q}}$ and $\mathbf{1}$.
\end{cor}

Using Theorem~\ref{th:Fgraph} we can describe the case when $X$ is finite
(not necessarily irreducible).
Clearly in this case $\bar {\mathbf{q}}(w)={\mathbf{q}}(w, A_w)$ where $A_w:=\{x \in X \colon w \to x\}$.
Moreover, 
for all $x \in X$ we have that
it is either ${\mathbf{q}}(\cdot, x)=\bar {\mathbf{q}}(\cdot)$ or there exists $w \in X$ such that ${\mathbf{q}}(w,x)=1$. If the BRW is
irreducible (and $X$ is finite) then it is 
$\bar {\mathbf{q}}(w)={\mathbf{q}}(w,w)$ for all $w \in X$ or ${\mathbf{q}}(w,x)=1$ for all $w,x \in X$.

\begin{cor}\label{pro:qtransitive}
 Let $(X, \mu)$ be an irreducible and quasi-transitive BRW.
Then the existence of $x \in X$ such that there is local survival at $x$ (i.e.~${\mathbf{q}}(x,x) < 1$)
implies that there is strong local survival at $y$ starting from $w$ for
every $w,y \in X$ (i.e~${\mathbf{q}}(w,y)=\bar {\mathbf{q}}(w)$).
\end{cor}

Hence for a quasi-transitive, irreducible BRW, whenever there is local
survival, it is a strong local survival; in continuous-time this implies that
there is global and local extinction if $\lambda \in [0, \lambda_w]$,
pure global survival if $\lambda \in (\lambda_w, \lambda_s]$ and
strong local survival if $\lambda \in (\lambda_s, +\infty)$
(see also Theorem~\ref{th:continuoussurv}).

In the particular case of a quasi-transitive, irreducible BRW with no death
and with independent diffusion, 
Corollary~\ref{pro:qtransitive} was proved in~\cite[Theorem 3.7]{cf:Muller08-2}.
The proof we give in Section~\ref{sec:proofs} is of a different nature.
Unlike Theorem~\ref{th:Fgraph}, Corollary~\ref{pro:qtransitive} does not hold for every
$\mathcal F$-BRW; indeed, as Examples~\ref{ex:nonstronglocalFBRW} and \ref{ex:nonstronglocalFBRW2}
show, for an irreducible $\mathcal F$-BRW there might be non-strong local survival.

The following result follows by applying \cite[Theorem 3.1]{cf:MenshikovVolkov} to
the no-death BRW associated to a generic BRW as described in
Section~\ref{subsec:genfun} (hence we omit the proof).
The original result \cite[Theorem 3.1]{cf:MenshikovVolkov} can be recovered from this one by assuming
$\rho_x(0)=0$ for all $x \in X$ which implies that $\bar {\mathbf{q}}={\mathbf{0}}$ and  
$T^{-1}_{\bar {\mathbf{q}}}$ is equal to the identity.

\begin{teo}
\label{th:MenshikovVolkov} 
Let $(X,\mu)$ be an irreducible, globally surviving BRW. 
Then there is no strong local survival
if and only if  there exists a finite, nonempty set $A \subseteq X$ and a function ${\mathbf{v}} \in [0,1]^X$ such that
$\bar {\mathbf{q}} \le {\mathbf{v}}$ and
\[ 
 \begin{cases}
  G({\mathbf{v}}|x) \ge {\mathbf{v}}(x), & \forall x \in A^\complement,\\
  (T^{-1}_{\bar {\mathbf{q}}}{\mathbf{v}})(x_0) > \max_{x \in A} (T^{-1}_{\bar {\mathbf{q}}}{\mathbf{v}})(x) & \textrm{for some } x_0 
  \in A^\complement,
 \end{cases}
\] 
where $T^{-1}_{\bar {\mathbf{q}}}{\mathbf{v}}=({\mathbf{v}}-\bar {\mathbf{q}})/(\mathbf{1}-\bar {\mathbf{q}})$.
\end{teo}

\subsection{Pure global survival}
\label{subsec:stronglocal}

%
%
The idea of \textit{pure global survival} (see Definition~\ref{def:survival}(4)) 
has been first introduced in continuous-time
BRW theory (and, more generally, in interacting
particle theory) to define the situation where $\lambda_s(x)>\lambda_w(x)$. In this case
for every $\lambda \in (\lambda_w(x), \lambda_s(x)]$ there is a positive probability
of global survival starting from $x$ but the colony dies out locally at $x$ almost surely.
A necessary condition for the existence of a pure global survival phase starting from $x$
is $K_s(x,x)<K_w(x)$ (see Theorem~\ref{th:continuoussurv}).
According to Theorem~\ref{th:continuoussurv}(3), for an $\mathcal{F}$-BRW this condition is also sufficient.

Clearly for an irreducible, continuous-time BRW, the existence of pure global survival does not depend
 on the starting vertex since
$\lambda_w(x)=\lambda_w$ and $\lambda_s(x)=\lambda_s$ for all $x \in X$.
This is still true for an irreducible discrete-time BRW as a consequence of
Proposition~\ref{pro:maximumprinciple}. Indeed, if ${\mathbf{z}}(\cdot):={\mathbf{q}}(\cdot,A)$
we have that $\widehat {\mathbf{z}}(x)$ can be interpreted as the probability of
local extinction in $A$ conditioned on global survival (starting from $x$).
Thus, according to Proposition~\ref{pro:maximumprinciple}, 
if the BRW is irreducible, then this conditional probability is one everywhere, provided it
is one somewhere. This means that if there is pure global survival starting from some $x$
then there is pure global survival starting from every $x$. 

Theorem~\ref{th:continuoussurv}(4) tells us that
an irreducible, continuous-time $\mathcal{F}$-BRW has a pure global survival phase 
if and only if it is nonamenable. This is not true if the process is not an $\mathcal{F}$-BRW
as shown by Example~\ref{exm:amenable}.
The same example shows that pure global survival is a fragile property of a BRW. Indeed, finite modifications,
such as for an edge-breeding BRW attaching a complete finite graph to a vertex
or removing a set of
vertices and/or edges,
 can create it or destroy it.

\section{Examples}\label{sec:examples}

 The first example shows that there are irreducible amenable BRWs with pure global survival and
irreducible nonamenable BRWs with no
pure global survival (see also \cite{cf:PemStac1}).

\begin{exmp}\label{exm:amenable}
In this example we use many times the following argument (which is an adaptation from \cite[Remark 3.2]{cf:BZ}).
Consider a continuous-time BRW adapted to a connected graph $X$, in the sense that
 $k_{xy}>0$ if and only if $(x,y)$ is an edge.
In some cases it is easy to show that  the existence of a pure global survival of a BRW
implies the existence of a pure global survival of the BRW restricted to some subgraph
(where all the rates $k_{xy}$ are turned to $0$ if $x$ or $y$ do not belong to the subgraph). 
Indeed
if $Y$ is a finite subset of $X$ such that $X \setminus Y$ is divided into
a finite number of connected graphs $X_1, \ldots, X_n$ (which is certainly true if
$k_{xy}>0$ is equivalent to $k_{yx} >0$ for all $x,y \in X \setminus Y$), then for every
$\lambda \in (\lambda_w^{X},\lambda_s^{X})$ the BRW on $X$ leaves eventually a.s.~the subset $Y$.
Hence it survives (globally but not locally) at least on one connected component; this means that,
although $\lambda_s^{X_i} \ge \lambda_s^{X}$, $\lambda_w^{X_i} \ge \lambda_w^{X}$ for all $i=1, \ldots,n$
(since $X_i \subseteq X$),
there exists $i_0$ such that $\lambda_w^{X_{i_0}} = \lambda_w^{X}$. The existence of a pure
global survival on $X_{i_0}$ follows from
$\lambda_s^{X_{i_0}} \geq \lambda_s^{X}>\lambda_w^{X}=\lambda_w^{X_{i_0}}$.
Hence if there exists a subset $Y$ as above such that $\lambda_w^{X_i} > \lambda_w^{X}$ for all $i$,
then there is no pure global survival for the BRW on $X$.
%

Consider an irreducible, edge-breeding continuous-time BRW on the (non-oriented) graph $X$ obtained by attaching to a copy of
 $\N$ one branch $T$ of the homogeneous tree $\mathbb T_3$. 
 The BRW is amenable by the presence of the copy of $\N$.
We claim that $\lambda_s^X=\lambda_s^{\mathbb T_3}$ and $\lambda_w^X = \lambda_w^{\mathbb T_3}$.
Indeed $T\subset X\subset \mathbb T_3$, hence
 $\lambda_s^T\ge \lambda_s^X\ge\lambda_s^{\mathbb T_3}$ and  $\lambda_w^T\ge
 \lambda_w^X \ge \lambda_w^{\mathbb T_3}$.  But by approximation, $\lambda_s^T=\lambda_s^{\mathbb T_3}$.
 Indeed $\lambda_s^T\ge \lambda_s^{\mathbb T_3}$ and does not depend on the starting vertex; moreover
 $T$ contains arbitrarily large balls isomorphic to balls of $\mathbb T_3$, hence by 
\cite[Theorem 5.2]{cf:Z1}\footnote{In \cite[Section 5.1]{cf:Z1}
the hypotheses that $M$ is a nonnegative matrix is missing, even though it is implicitly used.
Moreover \cite[Theorems 5.1 and 5.2]{cf:Z1} hold without the irreducibility hypothesis: the key is to note that,
given a sequence $\{X_n\}_{n \in \N}$ of subsets of $X$ such that $\liminf_n X_n=X$ and defined the sequence of matrices
$\{_nM\}_{n \in \N}$ as $_nM:=(m_{x,y})_{x,y \in X_n}$, one can
prove that, for all $x_0 \in X$, we have $(_nM)_s(x_0,x_0) \rightarrow M_s(x_0,x_0)$ (as defined
in equation~\eqref{eq:generalgeomparam}).}
or \cite[Theorem 3.1]{cf:BZ3}
 their critical local parameters coincide.
Explicit computations show that $\lambda_w^{\mathbb T_3}=1/3<1/2\sqrt{2}=\lambda_s^{\mathbb T_3}$ (there is pure global survival
on $\T_3$).
 Since $\T_3$ can be obtained by attaching three copies of $T$ to a root, the above discussion about surviving on a subgraph, implies
that $\lambda_w^T=\lambda_w^{\mathbb T_3}$.
Then we have
 $\lambda_w^X = \lambda_w^T \le \lambda_w^{\mathbb T_3}<\lambda_s^{\mathbb T_3}=\lambda_s^X$ and there is pure global survival on $X$.
%

On the other hand, consider a nonamenable graph $X^\prime$ such that the corresponding edge-breeding continuous-time
BRW has a pure global survival phase (take for instance $X^\prime:=\mathbb{T}_3$ the homogeneous tree
with degree $3$). 
Attach
to a vertex of $X^\prime$ a complete graph $Y$ with degree $k>1/\lambda_w^{X^\prime}$ by an edge.
It is easy to show that the resulting
graph $X$ is still nonamenable, nevertheless 
there
is no pure global survival for the corresponding edge-breeding BRW. Indeed, by the above discussion, if there were
pure global survival on $X$ then one of the connected components of $X \setminus Y$ should have the same global critical value
of $X$; but $X\setminus Y=X^\prime$ and $\lambda_w^X \le 1/k<\lambda_w^{X^\prime}$.
Roughly speaking, 
it happens that for every
$\lambda \in (\lambda_w^X,\lambda_w^{X^\prime})$ the process cannot survive globally in
$X^\prime$ 
hence it hits infinitely often with positive probability the complete
graph, thus $\lambda_s^X=\lambda_w^X$.
%
\end{exmp}

The following example shows that the strong local survival is not monotone.
The counterexample is obtained by modifying the edge-breeding
BRW on a particular graph, namely the homogeneous
tree $\mathbb{T}_d$. The crucial property that we need here is the existence of a pure global survival
phase, thus the procedure applies to every BRW with such a phase.

\begin{exmp}\label{rem:nonstrongandstrong}
Consider the edge-breeding continuous-time BRW on the homogeneous
tree $\mathbb{T}_d$ with degree $d\ge 3$. 
Since the graph has constant degree $d$, the BRW can be seen also as a site-breeding process
where $k(x)=d$ for all $x \in \mathbb{T}_d$.
Hence it is locally isomorphic to a branching process which implies that $\lambda_w(x)=1/K_w(x)=1/d$ for all $x \in \mathbb{T}_d$
and if $\lambda \le 1/d$ then the probabilities of survival are $0$ (see Theorem~\ref{th:continuoussurv}(3)).
Similarly, according to Theorem~\ref{th:continuoussurv}(1), $\lambda_s(x)=1/K_s(x,x)$ which does not depend on $x$. 
By the definition of $P$ and the discussion after equation~\eqref{eq:counterpart}, 
we have that $K_s(x,x) = d \cdot \limsup_{n \to \infty} \sqrt[n]{p^{(n)}(x,y)}$
where $P$ is the diffusion matrix of the simple random walk on $\mathbb{T}_d$. 
Using \cite[Lemma 1.24]{cf:Woess}, we obtain $K_s(x,x)=2 \sqrt{d-1}$ which implies $\lambda_s(x)=1/2\sqrt{d-1}$ for all 
$x \in \mathbb{T}_d$ (and there is global extinction when $\lambda=\lambda_w$).
Hence, 
if $\lambda > 1/2\sqrt{d-1}$ there is strong local survival (see Corollary~\ref{pro:qtransitive})
while 
if $\lambda \in (1/d ,1/2 \sqrt{d-1}]$ the probability
of global survival is positive and
independent of the starting point and the probability of local survival at any finite $A \subseteq X$
is $0$.

Fix $\lambda \in (1/d, 1/2 \sqrt{d-1}]$ and a finite  $A \subseteq X$. According to Theorem~\ref{rem:strongconditioned},
there exists  $x \in X$  such that
there is a positive probability of global survival starting from $x$
without ever visiting $A$ (clearly $x\not\in A$). In this case, any modification of the rates in the subset $A$ provides 
a new BRW such that 
there is still a positive probability of global survival starting from $x$
without ever visiting $A$ (since, the original BRW and the new one coincide until the first hitting time on $A$).
On the other hand, if there is $y \in A$ such that $x \to y$ and we add a loop in $y$
and a rate $k_{yy}>1/\lambda$ then $\bar {\mathbf{q}}(x)<{\mathbf{q}}(x,y)<1$;
the first inequality holds by the above discussion on local modifications
and the second one holds since $\lambda k_{yy}>1$ implies local survival at $y$
(then irreducibility implies local survival at $y$ starting from $x$).
This means that, for this fixed value of $\lambda$, we obtained a locally and globally
(but not strong-locally) surviving BRW at $y$ starting from $x$.

Suppose now that $k_{yy}>d$;
then, as in Example~\ref{exm:amenable},
we have a new BRW such that $\lambda_w^\prime=\lambda_s^\prime \le 1/k_{yy}$. In this case, when $\lambda \le \lambda_w^\prime$
there is global extinction.
When $\lambda > 1/2\sqrt{d-1}$ there is  strong local survival for the new BRW since 
there is strong local survival for the original one
(the probability of hitting $x$ conditioned on global survival is $1$ for both
processes and Theorem~\ref{rem:strongconditioned} applies).
If $\lambda \in (\lambda_w^\prime, 1/d]$ there is local and global survival
with the same probability since in order to survive globally,
the process must visit $x$ infinitely many times (it cannot survive globally in the branches of $\mathbb{T}_d$).
If $\lambda \in (1/d, 1/2\sqrt{d-1}]$ then $k_{yy}>1/\lambda$ and, according to the previous discussion,
there is non-strong local survival for the new BRW.
\end{exmp}

We show that even in the irreducible case, if $\rho_x(0)=0$ for some $x \in X$, we might have
strong local survival starting from some vertices and not from others.

\begin{exmp}\label{exm:irreduciblenotstrongeverywhere}
 Let us consider a modification of the discrete-time counterpart of the edge-breeding BRW on
$\mathbb{T}_d$ with degree $d\ge 3$ and $\lambda \in (1/d, 1/2\sqrt{d-1}]$ . Let us fix
a vertex $y$; in this modified version we add, with probability one,
one child at $y$ for every particle at $y$. In this case $\bar {\mathbf{q}}(y)={\mathbf{q}}(y,A)=0$ for all $A \subseteq X$.
On the other hand as  in Example~\ref{rem:nonstrongandstrong},
there is a vertex $x$ such that $\bar {\mathbf{q}}(x)<{\mathbf{q}}(x,y)$.
\end{exmp}
\smallskip

In the last few examples we make use of the subclass of BRWs which are
locally isomorphic to a branching process (which are particular
$\mathcal F$-BRWs, see Section~\ref{subsec:FBRWs}). 
By using Theorems~\ref{th:discretesurv} and \ref{th:continuoussurv} and the explicit
computations for $M_s$ and $M_w$ given after equation~\eqref{eq:generalgeomparam},
it is easy to show that for such a process: (1) there is global survival if and only if $\bar \rho>1$;
(2) there is local survival at $x$
if and only if $\bar \rho > 1/\limsup_{n \to \infty} \sqrt[n]{p^{(n)}(x,x)}=:r(x,x)$.
Hence, 
in the irreducible case, there is pure global survival if and only if
$1< \bar \rho \le r$ (where $r=r(x,x)$ in this case does not depend on $x \in X$ due to irreducibility).
This is possible if and only if $r>1$ which is equivalent to nonamenability 
since in this case
$
M_s(x,y)= \bar \rho / r$ and
$M_w(x) 
=\bar \rho$.
%
It is clear that, given a
continuous-time BRW which is locally isomorphic to a branching process,
$\lambda_w=1/k$ and $\lambda_s(x)=r(x,x)/k$ (where $k=k(x)$ for all $x \in X$).

In general there may be non-strong local survival, even if the BRW is irreducible, 
locally isomorphic to a branching process
and it has independent diffusion 
as Examples~\ref{ex:nonstronglocalFBRW} and \ref{ex:nonstronglocalFBRW2} show.
This disproves \cite[Theorem 3 and Corollary 4]{cf:Spataru89} (see also Remark~\ref{rem:Spataru}) since
$\bar{\mathbf{q}} < \mathbf{q}(\cdot,0) < \mathbf{1}$ are three distinct fixed points of $G$.
%

\begin{exmp}\label{ex:nonstronglocalFBRW}
Fix $X:=\N$ and consider a BRW with the following reproduction probabilities.
Every particle has two children with probability $3/4$ and no children with probability $1/4$. Each newborn
particle is dispersed independently according to a nearest neighbor matrix $P$ on $\N$.
More precisely
\[
 p(i,j):=
\begin{cases}
 p_i & \textrm{ if } j=i+1\\
1-p_i  & \textrm{ if } j=i-1,\\
\end{cases}
\]
and $p_0=1$.
The process described above is
an irreducible $\mathcal F$-BRW for every choice of the set $\{p_i\}_{i \in \N \setminus \{0\}}$
such that $p_i \in (0,1)$ for all $i>0$.
The generating function of the total number of children is $z \mapsto 3 z^2/4+1/4$ and its minimal
fixed point is $1/3=\bar {\mathbf{q}}(x)$ (for all $x \in \N$).

Choose $p_1 < 5/9$; it is easy to show that the process confined to $\{0,1\}$ (that is, every particle sent outside $\{0,1\}$ is
killed) survives,
since the expected number of children at $0$ every two generations (starting from $0$)
is $ (3/2)^2 (1-p_1)>1$. Since the confined process is stochastically dominated by the original one,
we have local survival, for instance, at $x=0$.
By irreducibility this implies that ${\mathbf{q}}(x,y)<1$ and $\bar {\mathbf{q}}(x)<1$ for all $x,y \in \N$.

Choose the $p_i$s such that $\prod_{i=1}^\infty p_i^{2^i}>0$ (or, equivalently, 
 $\sum_{i=1}^\infty 2^i (1-p_i)< +\infty$).
Consider the branching process $N_n$ representing the total number of particles alive at time $n$:
for all $n$,  $N_n \le 2^n$ almost surely.
The probability, conditioned on global survival, that every particle places its children (if any) to its right,
is the conditional expected value of $\prod_{i=1}^\infty p_i^{N_i}$.
But
$\prod_{i=1}^\infty p_i^{N_i} \ge \prod_{i=1}^\infty p_i^{2^i}>0$ almost surely.
 Hence, conditioning on global survival
there is a positive probability of non-local survival. This implies ${\mathbf{q}}(\cdot, y) \not = \bar {\mathbf{q}}$
for every $y \in \N$. Note that, according to Theorem~\ref{th:Fgraph}, $\sup_{x \in \N} {\mathbf{q}}(x,x)=1$.
This proves that, even in the irreducible case, the generating function $G$ can have more than two
fixed points (see also Remark~\ref{rem:Spataru}).
\end{exmp}

The key in the previous example is that the total number of particles alive at time $n$ is bounded.
This is not an essential assumption. The following example shows that, given any law $\rho$ of a
surviving
branching process (that is, $\bar \rho=\sum_{n \in \N} n \rho(n)>1$), it is possible to construct
an irreducible BRW which is locally isomorphic to a branching process with non-strong local survival.

\begin{exmp}\label{ex:nonstronglocalFBRW2}
Let $X=\N$ and $\rho_x:=\rho$ for all $x \in \N$; $\rho$ being the law of a surviving branching process.
We know that $\bar {\mathbf{q}}(x) \equiv \bar q$ for all $x \in \N$
where $\bar q<1$ is the smallest fixed point of $z \mapsto \sum_{n\ in \N} \rho(n) z^n$.
Pick a sequence of natural numbers
$\{N_i\}_{i \in \N}$ satisfying
\begin{equation} \label{eq:defineNi}
\prod_{i \in \N} \rho([0,N_{i+1}])^{\prod_{j=0}^i N_j} >\bar q,
\end{equation}
where $N_0:=1$.
Note that the probability of the event $\mathcal{A}$=``every particle alive at time $i$ has at most $N_{i+1}$ children
for all $i \in \N$'' is bounded from below by the LHS of equation~\eqref{eq:defineNi}.
Thus, from equation~\eqref{eq:defineNi}, with a probability larger than
$\prod_{i \in \N} \rho([0,N_{i+1}])^{\prod_{j=0}^i N_j} -\bar q>0$ the colony survives globally and
the total size of the population at time $n$ is not larger than $\prod_{j=0}^n N_j$
(i.e.~the intersection between $\mathcal{A}$ and global survival has positive probability).

We define a BRW with independent diffusion 
where $P$ is as follows
\[
 p(i,j)
:=
\begin{cases}
 p_i &  j=i+1,\, i \ge 0\\
1-p_i & j=i-1, \, i \ge 1\\
1-p_0 & i=j=0.\\
\end{cases}
\]
Let $p_0$ such that $(1-p_0)\bar \rho>1$; this implies local survival.
We choose the sequence $\{p_i\}_{i \in \N}$, where $p_i \in (0,1)$ in such a way that
\begin{equation} \label{eq:definepi}
\prod_{i \in \N} p_i^{\prod_{j=0}^i N_j} >0
\end{equation}
(or, equivalently, 
$\sum_{i \in \N} (1-p_i){\prod_{j=0}^i N_j}< \infty$).
Using equation~\eqref{eq:definepi}, if we condition on $\mathcal A$,
the probability that, every particle places its children (if any) to its right is bounded from
below by $\prod_{i \in \N} p_i^{\prod_{j=0}^i N_j}$.
This implies that there is a positive probability of global, non-local survival.

The choice of the sequences $\{N_i\}_{i \in \N}$ and $\{p_i\}_{i \in \N}$ satisfying equations~\eqref{eq:defineNi} and
\eqref{eq:definepi} respectively can be done as follows.
Choose a sequence $\{\alpha_i\}_{i \in \N}$ such that $\alpha_i \in (0,1)$ for all $i \in \N$ and
$\prod_{i \in \N} \alpha_i >1-\bar q$. Then, iteratively, if we fixed $N_0, \ldots, N_k$, since $\lim_{x \to \infty}\rho([0,x])=1$
there exists $N_{k+1} \in \N$ such that $\rho([0,N_{k+1}])>\alpha_{k+1}^{1/\prod_{j=0}^k N_j}$.
Let us take, for instance, $p_i > 1/(i \cdot {\prod_{j=0}^i N_j})$.

We note that the class constructed in this example includes
discrete-time counterparts of continuous-time BRWs where
$\rho$ can be chosen as in equation~\eqref{eq:counterpart} where $k(x) \equiv k$ does not depend on $x$,
$k_{xy}:=k \cdot p(x,y)$ (where $P$ is defined as before) and $\lambda>\lambda_s$ is fixed.
Finally we observe that this example extends naturally to an example
of a site-breeding BRW on a radial tree where the number of branches of a vertex
at distance $k$ from the root is at least $1/p(k,k+1)$.
\end{exmp}

Even though the local extinction probability ${\mathbf{q}}(\cdot,y)$ (for any fixed $y \in \mathbb{N}$)
of Examples~\ref{ex:nonstronglocalFBRW} and \ref{ex:nonstronglocalFBRW2} provides
a fixed point which is different from both $\overline {\mathbf{q}}$ and $\mathbf{1}$ for the function $G$
of an irreducible BRW, in the following example we give a more explicit construction of such a fixed point.
\begin{rem}\label{rem:Spataru}
Consider a generating function 
\begin{equation}\label{eq:explicit}
 G({\mathbf{z}}|n)=
\begin{cases}
 3 \big (p_n {\mathbf{z}}(n+1)+(1-p_n){\mathbf{z}}(n-1) \big )^2/4+1/4 & n \ge 1 \\
 3 {\mathbf{z}}(1)^2/4+1/4 & n=0,
\end{cases}
\end{equation}
on $[0,1]^\mathbb{N}$ where $p_n \in (0,1)$ for all $n \in \mathbb{N}$.
This is the generating function of an irreducible BRW and
the constant vectors $\overline {\mathbf{q}} \equiv 1/3$ and $\mathbf{1}$ 
are always fixed points of $G$ regardless of the choice of $\{p_n\}$.
An explicit construction of $\{p_n\}$ and of a third fixed point ${\mathbf{z}} \in (0,1)^X$ can be 
carried out recursively as follows. Take ${\mathbf{z}}(0) \in (1/3, 1)$,  $p_0=1$ and $p_1<5/9$.
The explicit expression of the equation $G({\mathbf{z}})={\mathbf{z}}$ is easily derived from equation~\eqref{eq:explicit}.
  Since $1>\sqrt{({4x-1})/{3}}> x$ for all $x \in (1/3,1)$ then $1>{\mathbf{z}}(1)>{\mathbf{z}}(0)>1/3$. 
Suppose that $1>{\mathbf{z}}(n) > {\mathbf{z}}(n-1) > \cdots >{\mathbf{z}}(0)>1/3$
for some $p_0, p_1, \ldots, p_{n-1} \in (0,1)$.
Choose ${\mathbf{z}}(n+1)<1$ such that $3 {\mathbf{z}}(n+1)^2/4+1/4 >{\mathbf{z}}(n)$. 
By continuity, there exists $p_n<1$ such
that $3 \big (p_n {\mathbf{z}}(n+1)+(1-p_n){\mathbf{z}}(n-1) \big )^2/4+1/4={\mathbf{z}}(n)$.
By induction we have a new fixed point ${\mathbf{z}}$ of this $G$ (associated to the sequence 
$\{p_n\}$) such that $1/3<{\mathbf{z}}(n)<{\mathbf{z}}(n+1)<1$ for all $n$
and $\lim_{n \to \infty} p_n=\lim_{n \to \infty} {\mathbf{z}}(n)=1^-$. 

This disproves \cite[Theorem 3 and Corollary 4]{cf:Spataru89}. Indeed there is a
gap in the proof of 
\cite[Theorem 3]{cf:Spataru89}: in the line 8 of the proof, the sentence ``Clearly 
$B^\prime \not = \emptyset$'' is incorrect when the set $X$ is infinite as the following example shows.
%
%
Using the notation of \cite{cf:Spataru89} take ${\mathbf{x}}, {\mathbf{t}} \in [0,1]^X$ such that
${\mathbf{x}}$ is constant, say ${\mathbf{x}}(i)=a<1$ for all $i \in X$ and $1>{\mathbf{t}}(i)>a$ for all $i \in X$. Suppose that
$\sup_{i \in X} {\mathbf{t}}(i)=1$. Then the half-line $\{{\mathbf{x}} + \theta({\mathbf{t}}-{\mathbf{x}})\colon \theta \ge 0\}$ exits from
the set $[0,1]^X$ at $\theta=1$, that is, at the  point ${\mathbf{t}}$. 
Indeed if $\theta >1$ then $\sup_{i \in X} ({\mathbf{x}}(i)+\theta({\mathbf{t}}(i)-{\mathbf{x}}(i)))
=\theta \sup_{i \in X} {\mathbf{t}}(i) -a(\theta-1)=\theta -a(\theta-1)=1+(\theta-1)(1-a)>1$.
But $B^\prime:=\{ i \in X \colon {\mathbf{t}}(i)=1\}=\emptyset$; roughly speaking, in this case 
there is not a smallest
value for $\theta \ge 0$ such that some coordinates of the point ${\mathbf{x}}+\theta({\mathbf{t}}-{\mathbf{x}})$ are $1$.  
\end{rem}

\section{Proofs}
\label{sec:proofs}

\begin{proof}[Proof of Proposition~\ref{pro:maximumprinciple}]
If $\bar {\mathbf{q}}=\mathbf{1}$ there is nothing to prove. Suppose that $\bar {\mathbf{q}}<\mathbf{1}$.
Without loss of generality we can suppose that $\bar {\mathbf{q}}(x)<1$ for all $x \in X$. Indeed,
given $x_0$ such that $\bar {\mathbf{q}}(x_0)=1$ then for all $x \in \mathcal{N}_{x_0}$ we have
$\bar {\mathbf{q}}(x)=1$. Since we defined $\widehat {\mathbf{z}}(x):=1$ whenever $\bar {\mathbf{q}}(x)=1$ we can remove
these vertices obtaining a new set $X^\prime \subseteq X$.
Consider the restricted BRW on $X^\prime$ (obtained by killing all the particles going outside $X^\prime$).
It is clear that
${\mathbf{q}}^X(x,A)\le {\mathbf{q}}^{X^\prime}(x,A)$ for all $x\in X^\prime$, $A\subseteq X^\prime$.
The generating function $G^\prime$ of the new BRW satisfies $G^\prime(({\mathbf{z}}_{|{X^\prime}})|x) \ge G({\mathbf{z}}|x)$ for all
$x \in X^\prime$, hence $G({\mathbf{z}}) \ge {\mathbf{z}}$ implies $G^\prime({\mathbf{z}}_{|{X^\prime}}) 
\ge {\mathbf{z}}_{|{X^\prime}}$ (where ${\mathbf{z}}_{|{X^\prime}}$ is ${\mathbf{z}}$ restricted to $X^\prime$).
Moreover $\widehat {\mathbf{z}}$ satisfies the conclusions of the proposition if and only 
if $\widehat{{\mathbf{z}}_{|{X^\prime}}}\equiv
 \widehat {\mathbf{z}}_{|{X^\prime}}$
does. Thus, it is enough to prove the result for the BRW restricted to $X^\prime$.

Note that $\widehat {\mathbf{z}}:=T_{\bar {\mathbf{q}}}^{-1}({\mathbf{z}})$, 
thus $G({\mathbf{z}}) \ge {\mathbf{z}}$  is equivalent to $\widehat G(\widehat {\mathbf{z}}) \ge \widehat {\mathbf{z}}$.
Hence it is enough to prove the proposition when $\mu_x(\mathbf{0})=0$ for all $x \in X$
which implies $\bar {\mathbf{q}}=\mathbf{0}$ and $\widehat {\mathbf{z}}= {\mathbf{z}}$.
Suppose that $\mathcal{N}_x$ is nonempty, ${\mathbf{z}}(y) \le {\mathbf{z}}(x)$ for all $y \in  \mathcal{N}_x$
and ${\mathbf{z}}(y_0)< {\mathbf{z}}(x)$ for some $y_0 \in  \mathcal{N}_x$.
Then, using the fact that ${\mathbf{z}} \le \mathbf{1}$ and that $\prod_{y \in X} {\mathbf{z}}(y)^{f(y)} \le {\mathbf{z}}(x)$ if $\mathcal{H}(f) \ge 1$,
we have that
${\mathbf{z}}(x) \le G({\mathbf{z}}|x) \le \sum_{f \in S_X\colon f(y_0)=0} \mu_x(f) {\mathbf{z}}(x) +\sum_{f \in S_X\colon f(y_0)>0} \mu_x(f) {\mathbf{z}}(y_0)<  {\mathbf{z}}(x)$
which is a contradiction.
As for the second part, since ${\mathbf{z}}(y) \le 1= {\mathbf{z}}(x)$ for all $y \in X$ then we have ${\mathbf{z}}(y)=1$ for all $y \in X$.
Finally, by induction we obtain the result for the set $\{y \in X\colon x \to y\}$.
\end{proof}
%
%

\begin{proof}[Proof of Theorem~\ref{rem:strongconditioned}]
 Indeed,
since $\{{\mathbf{q}}_n(\cdot,A)\}_{n \in\N}$ is non decreasing, ${\mathbf{q}}_{n}(\cdot,A)=G({\mathbf{q}}_{n-1}(\cdot, A))$ and
$\bar {\mathbf{q}}$ is the smallest fixed point of $G$, we have immediately that
\begin{equation}\label{eq:strongconditioned}
{\mathbf{q}}(\cdot, A) = \bar {\mathbf{q}}(\cdot) \Longleftrightarrow {\mathbf{q}}_0(\cdot,A) \le \bar {\mathbf{q}}(\cdot),
\end{equation}
that is, (1)$\Longleftrightarrow$(2).
Moreover
the event ``local survival in $A$ starting from $x$'' implies both ``global survival starting from $x$''
and ``visiting $A$ at least once starting from $x$'', hence
${\mathbf{q}}(x, A) = \bar {\mathbf{q}}(x)<1$ if and only if
the probability of visiting $A$ infinitely many times starting from $x$ conditioned on global survival
is $1$ and (1)$\Longleftrightarrow$(5)$\Longrightarrow$(4). Trivially (2)$\Longleftrightarrow$(3)
and (4)$\Longrightarrow$(3). This proves the equivalence.
\end{proof}

Before proving
Corollary~\ref{pro:qtransitive} and Theorem~\ref{th:Fgraph}
we need two technical lemmas.

\begin{lem}\label{lem:convexity}
 Let $(X,\mu)$ be a BRW and fix ${\mathbf{z}},v \in [0,1]^X$ such that ${\mathbf{z}}+\varepsilon v \in [0,1]^X$ for some
$\varepsilon>0$.
Then the function $t \mapsto G({\mathbf{z}}+vt|x)$ is strictly convex if and only if
\begin{equation}\label{eq:supp}
 \exists f \colon  \mu_x(f)>0 ,\, \sum_{y \in \mathrm{supp}(v)} f(y) \ge 2, \,
\mathrm{supp}({\mathbf{z}}) \cup \mathrm{supp}(v)\supseteq \mathrm{supp}(f).
\end{equation}
\end{lem}


\begin{proof}[Proof of Lemma~\ref{lem:convexity}]
Let us evaluate the function $G$ on the line $t \mapsto {\mathbf{z}}+tv$
where $t \in [0, T)$ and $T:=\sup\{s>0:{\mathbf{z}}+sv \in [0,1]^X\}$.

\[ 
\begin{split}
G({\mathbf{z}}+tv|x)&=
\sum_{f \in S_X} \mu_x(f) \prod_{y \in X} \sum_{i=0}^{f(y)} \binom{f(y)}{i} {\mathbf{z}}(y)^{f(y)-i}v(y)^i t^i\\
&=
\sum_{f \in S_X} \mu_x(f)  \sum_{g \in S_X\colon g \le f} \prod_{y \in X}\binom{f(y)}{g(y)} {\mathbf{z}}(y)^{f(y)-g(y)}v(y)^{g(y)} t^{g(y)}\\
&=
\sum_{f \in S_X} \mu_x(f)  \sum_{g \in S_X\colon g \le f} t^{\mathcal H(g)} \prod_{y \in X}\binom{f(y)}{g(y)} {\mathbf{z}}(y)^{f(y)-g(y)}v(y)^{g(y)}  \\
&=
\sum_{f \in S_X} \mu_x(f)  \sum_{i=0}^{\infty} \sum_{g \in S_X\colon \mathcal H (g)=i, g \le f}
t^{i} \prod_{y \in X}\binom{f(y)}{g(y)} {\mathbf{z}}(y)^{f(y)-g(y)}v(y)^{g(y)}  \\
&=
\sum_{i=0}^\infty t^{i} \left ( \sum_{f,g \in S_X\colon \mathcal H (g)=i, g \le f}
\mu_x(f)  \prod_{y \in X}\binom{f(y)}{g(y)} {\mathbf{z}}(y)^{f(y)-g(y)}v(y)^{g(y)} \right ) \\
\end{split}
\] 
The strict convexity of a power series in $t$ with nonnegative coefficients is equivalent to the strict positivity
of at least one coefficient corresponding to $t^i$ with $i \ge 2$.
Hence it is easy to show that each of the following assertions is equivalent to the next one and
that they are all equivalent to the strict
convexity of $t \mapsto G({\mathbf{z}}+vt|x)$
\begin{enumerate}
 \item $\exists f,g \colon  \mathcal H(g) \ge 2, \, f \ge g, \, \mu_x(f)>0 , \mathrm{supp}(v) \supseteq \mathrm{supp}(g), \,
\mathrm{supp}({\mathbf{z}}) \supseteq \mathrm{supp}(f-g)$;
 \item $ \exists f,g \colon  \mathcal H(g) \ge 2, \, f \ge g, \, \mu_x(f)>0 , g=f \ident_{\mathrm{supp}(v)}, \,
\mathrm{supp}({\mathbf{z}}) \supseteq \mathrm{supp}(f) \setminus \mathrm{supp}(v)$;
 \item $\exists f \colon  \mu_x(f)>0 , \sum_{y \in \mathrm{supp}(v)} f(y) \ge 2, \,
\mathrm{supp}({\mathbf{z}}) \supseteq \mathrm{supp}(f) \setminus \mathrm{supp}(v)$;
 \item $\exists f \colon  \mu_x(f)>0 , \sum_{y \in \mathrm{supp}(v)} f(y) \ge 2, \,
\mathrm{supp}({\mathbf{z}}) \cup \mathrm{supp}(v)\supseteq \mathrm{supp}(f)$;
\end{enumerate}
\end{proof}

\begin{lem}\label{lem:generationn}
 Let $(X, \mu)$ be a BRW and fix $x_0 \in X$. Suppose that for some
$\bar x$ in the same irreducible class of $x_0$ and $f \in S_X$ we have that $\mu_{\bar x}(f)>0$,
$\sum_{w\colon w \rightleftharpoons x_0} f(w) \ge 2$.
We can fix $\bar n \in\N$ such that
if the process starts with one particle at $x_0 \in X$ then we have at least 2 particles at
$x_0$ in the generation $\bar n$ wpp.
\end{lem}

\begin{proof}[Proof of Lemma~\ref{lem:generationn}]
Consider a path $x_0, x_1, \ldots, x_m=\bar x$ and let $f \in S_X$ be
such that $\mu_{\bar x}(f)>0$ and $\sum_{w\colon w \rightleftharpoons x_0} f(w) \ge 2$.
We can have two cases.

\smallskip

\noindent \textbf{(a)}.~There exists $x_{m+1} \in X$ such that $x_{m+1} \rightleftharpoons x_0$ and $f(x_{m+1}) \ge 2$; in this case consider the closed path
$x_0, x_{1}, x_{2}, \ldots, x_m, x_{m+1}, \ldots, x_n=x_0$ and take $\bar n:=n$. Since any particle at $x_i$
has at least one child at $x_{i+1}$ wpp and a particle at $\bar x$ has at least 2 children at $x_{m+1}$ wpp,
then any particle at $x_0$ has at least 2 descendants
at $x_0$ in the $\bar n$th generation.
Indeed, denote by $f_i \in S_X$ such that $\mu_{x_i}(f_i)>0$, $f_i(x_{i+1}) \ge 1$ for all $i =0, \ldots \bar n-1$
($f_m$ being $f$), then the probability  that a particle at $x_0$ has at least 2 particles at $x_0$
in the $\bar n$th generation is bounded from below by $\prod_{i=0}^m \mu_i(f_i) \prod_{j=m+1}^{\bar n-1} \mu_j(f_i)^2$.
\smallskip

\noindent \textbf{(b)}.~There exists a couple of different vertices
$x_{m+1}, y_{m+1}$ 
such that $x_{m+1}, y_{m+1}\rightleftharpoons x_0$ and \break
$f(x_{m+1}), f(y_{m+1}) \ge 1$; in
this case consider the paths $x_0, x_1, \ldots x_m, x_{m+1},
\ldots, x_{n_1}=x_0$ and $x_0, x_1, \ldots x_m, y_{m+1}, \ldots,
y_{n_2}=x_0$ and take $\bar n:= GCD(n_1,n_2)$ (the conclusion is
similar as before).
\end{proof}

\begin{proof}[Proof of Theorem~\ref{th:Fgraph}]
$(1)$. 
For every fixed point ${\mathbf{z}}$ of $G$, we know that ${\mathbf{z}} \ge \bar {\mathbf{q}}$ and ${\mathbf{z}} \le \mathbf 1_X$; this implies
that if $\sup_{x \in X} {\mathbf{z}}(x)<1$ for some fixed point then necessarily $\sup_{x \in X} \bar {\mathbf{q}}(x)<1$.
Hence, if $\bar {\mathbf{q}}=\mathbf 1$ there is nothing to prove. Otherwise,
we show that if $G({\mathbf{z}})={\mathbf{z}}$ and ${\mathbf{z}} \not = \bar {\mathbf{q}}$ then $\sup_{w \in X} {\mathbf{z}}(w)=1$.
Suppose that  the BRW is locally isomorphic to $(Y,\nu)$ through the map $g$
and define ${\mathbf{h}}(y):=\sup_{w \in g^{-1}(y)} {\mathbf{z}}(w)$.
Clearly ${\mathbf{h}} \in [0,1]^Y$ and ${\mathbf{h}} \circ g \ge {\mathbf{z}}$ which implies that
$G_Y({\mathbf{h}}) \ge {\mathbf{h}}$.
Indeed
\[
\begin{split}
G_Y({\mathbf{h}}|y)&= \sup_{x \in g^{-1}(y)}G_Y({\mathbf{h}}|g(x))=\sup_{x \in g^{-1}(y)}G({\mathbf{h}} \circ g|x) \\
&\ge
\sup_{x \in g^{-1}(y)}G({\mathbf{z}}|x)= \sup_{x \in g^{-1}(y)} {\mathbf{z}}(x) = {\mathbf{h}}(y).
\end{split}
\]
If $Y$ finite then we can choose $\widetilde y \in Y$ which minimizes
\[
 t(y):= \frac{1-\bar {\mathbf{q}}^Y(y)}{{\mathbf{h}}(y)-\bar {\mathbf{q}}^Y(y)}
\]
(where $t(y):=+\infty$ if ${\mathbf{h}}(y)=\bar {\mathbf{q}}^Y(y)$);
note that $t(y) \ge 1$ for all $y \in Y$ and $t(\widetilde y)<+\infty$.
By applying the maximum principle (Proposition~\ref{pro:maximumprinciple})
to the function $1/t(y)$ (where $y$ is ranging in the set $\{w\colon \bar {\mathbf{q}}^Y(w)<1\}$)
we have that it is constant on $\{y\colon \widetilde y \to y\}$.
Since $\bar {\mathbf{q}}^Y (\widetilde y)<1$ and $Y$ is finite,
then there exists $y_0$ such that $\widetilde y \to y_0$ and
there is local survival at $y_0$ starting from $y_0$.
Since $(Y,\nu)$ satisfies Assumption~\ref{assump:1}
then there exists $\bar y \rightleftharpoons y_0$ such that
a particle living at $\bar y$ wpp has at least 2 children in the
irreducible class of $y_0$.
Then by taking $y_0$ instead of $x_0$ in
Lemma~\ref{lem:generationn} we have that we can
find $\bar n \in \N$ such that the function
\[
\phi(t):=G_Y^{(\bar n)}(\bar {\mathbf{q}}^Y+t ({\mathbf{h}}-\bar {\mathbf{q}}^Y)|y_0)-\bar {\mathbf{q}}^Y(y_0)-t ({\mathbf{h}}(y_0)-\bar {\mathbf{q}}^Y(y_0))
\]
is strictly convex by Lemma~\ref{lem:convexity}.
Indeed $G_Y^{(\bar n)}$ is the generating function of the BRW constructed by considering
the $n$-th generations of the original BRW where $\bar n | n$ and, under our hypotheses, it satisfies
equation~\eqref{eq:supp}. 

Note that $\phi$ is well defined in $[0,t(y_0)]$ since
\[
\mathbf{r}_t(y):=\bar {\mathbf{q}}^Y(y)+t ({\mathbf{h}}(y)-\bar {\mathbf{q}}^Y(y)) \le \bar {\mathbf{q}}^Y(y)+t(y_0) ({\mathbf{h}}(y)-\bar {\mathbf{q}}^Y(y))\le 1
\]
hence $\mathbf{r}_t \in [0,1]^Y$ for all $t \in [0,t(y_0)]$.

Clearly every fixed point of $G_Y$ is a fixed point of $G_Y^{(\bar n)}$;
in particular, $G^{(\bar n)}({\mathbf{z}})={\mathbf{z}}$ and $G^{(\bar n)}_Y(\bar {\mathbf{q}}^Y)=\bar {\mathbf{q}}^Y$, whence
$\phi(0)=0$ and $\phi(1)=G^{(\bar n)}_Y({\mathbf{h}}|y_0)-{\mathbf{h}}(y_0)$.
Now, using equation~\eqref{eq:Gfunctions}, $G^{(\bar n)}_Y({\mathbf{h}}) \ge {\mathbf{h}}$
and this, in turn, implies $\phi(1) \ge 0$. 
Since $\phi$ is strictly convex we have that $\phi(t)>0$ for all $t \in (1, t(y_0)]$.
If $t(y_0)>1$ then $0 < \phi(t(y_0))=G^{(\bar n)}_Y({\mathbf{r}}_{t(y_0)}|y_0)-1$ but this is a contradiction
since ${\mathbf{r}}_{t(y_0)} \in [0,1]^Y$ and $G^{(\bar n)}_Y({\mathbf{r}}_{t(y_0)}) \in [0,1]^Y$.
In the end $t(y_0)=1$, thus $1={\mathbf{h}}(y_0)=\sup_{w \in X} {\mathbf{z}}(w)$.

$(2)$
This applies to $\mathbf{z}(\cdot)=\mathbf{q}(\cdot,x)$ for any fixed $x \in X$. If the BRW is irreducible 
$\mathbf{q}(w,x)= \mathbf{q}(w,y)$ for all $x,y,w \in X$. 
Thus, there exists $x \in X$ such that $\mathbf{q}(\cdot,x)=\bar{\mathbf{q}(\cdot)}$
if and only if  $\mathbf{q}(w,w)=\bar{\mathbf{q}(w)}$ for all $w \in X$; analogously,
there exists $x \in X$ such that $\sup_{w \in X} \mathbf{q}(w,x)=1$
if and only if  $\sup_{w \in X} \mathbf{q}(w,w)=1$.

\end{proof}
Note that, from the first part of the previous proof, 
if the BRW on $Y$ is irreducible then by the maximum principle 
we have
that $(\mathbf{h}-\bar {\mathbf{q}}^Y)/(\mathbf{1}-\bar {\mathbf{q}}^Y)$ is a constant function, thus 
$\mathbf{h}(y)=\sup_{w \in g^{-1}(y)} {\mathbf{z}}(w)=1$ for all $y \in Y$.

\begin{proof}[Proof of Corollary~\ref{cor:finiteBRW}]
 
When $X$ is finite, $(X, \mu)$ is clearly an $\mathcal{F}$-BRW.
If $\bar {\mathbf{q}}=\mathbf{1}$ 
there is nothing to prove. Suppose that $\bar {\mathbf{q}}<\mathbf{1}$,  since the BRW is irreducible
we have that $\bar {\mathbf{q}}(x)<{1}$ for all $x \in X$. 
Let $\bar{\mathbf{q}}<\mathbf{z}<\mathbf{1}$ be a solution of $G(\mathbf{z}) \ge \mathbf{z}$.
Since $X$ is finite and $\bar{\mathbf{q}}<\mathbf{z}$ from Theorem~\ref{th:Fgraph}(1) we have 
that $\mathbf{z}(x)=1$ for some $x \in X$. By Proposition~\ref{pro:maximumprinciple}, using irreducibility, 
$\widehat {\mathbf{z}}= \mathbf{1}$ which contradicts $\mathbf{z}<\mathbf{1}$.

\end{proof}

\begin{proof}[Proof of Corollary~\ref{pro:qtransitive}]
Since  $(X, E_\mu)$ is irreducible we have that ${\mathbf{q}}(x,y)={\mathbf{q}}(x,x)$ for all $x,y \in X$ and
if $\bar {\mathbf{q}}<\mathbf 1$ (resp.~${\mathbf{q}}(\cdot,y)<\mathbf 1$)
then $\bar {\mathbf{q}}(x) <1$ (resp.~${\mathbf{q}}(x,y)<1$) for all $x \in X$.
Moreover, quasi transitivity implies that if ${\mathbf{q}}(\cdot,y)<\mathbf 1$ then
$\sup_{x \in X} {\mathbf{q}}(x,y) <1$.
Thus, according to Theorem~\ref{th:Fgraph}, ${\mathbf{q}}(\cdot,y) \not = \mathbf 1$
implies ${\mathbf{q}}(\cdot,y)=\bar {\mathbf{q}}$.
\end{proof}

\section*{Acknowledgments}
The authors are grateful to the anonymous referee for carefully reading the manuscript and for useful suggestions
which helped to improve the paper.

\end{document}